\documentclass[11pt]{amsart}

\usepackage[normalem]{ulem}
\usepackage{url}
\usepackage{hyperref}
\usepackage{amssymb}
\usepackage{graphicx}
\usepackage{xspace}
\usepackage[usenames,dvipsnames]{color}
\usepackage{ifpdf}

\definecolor{brown}{cmyk}{0, 0.72, 1, 0.45}
\definecolor{grey}{gray}{0.5}

\newcommand{\ignore}[1]{\relax}

\newtheorem{theorem}{Theorem}
\newtheorem{lemma}[theorem]{Lemma}
\newtheorem{claim}[theorem]{Claim}

\newtheorem{remark}[theorem]{Remark}

\newtheorem{corollary}[theorem]{Corollary}
\newtheorem{definition}[theorem]{Definition}

\renewcommand{\l}{\lambda}
\renewcommand{\a}{\alpha}

\newcommand{\set}[1]{\{ #1 \}}
\newcommand{\card}[1]{\left| #1 \right|}

\newcommand{\bool}{\set{0,1}}

\newcommand{\bi}{\bigskip\noindent}
\newcommand{\si}{\smallskip\noindent}
\newcommand{\la}{\lambda}

\newcommand{\e}{\varepsilon}

\newcommand{\pr}{\mathbb{P}}
\renewcommand{\Pr}[1]{\mathbb{P} \! \left( {#1} \right)}
\newcommand{\ex}{\mathbb{E}}
\newcommand{\E}[1]{\mathbb{E}\left[ {#1} \right]}
\newcommand{\var}{\text{Var}}

\newcommand{\by}{\times}
\newcommand{\Po}{\operatorname{Po}}

\newcommand{\cstar}{c^*}
\newcommand{\cstark}{c_k^*}

\newcommand{\one}{\boldsymbol 1}
\newcommand{\zero}{\boldsymbol 0}
\newcommand{\Neqref}[1]{\textbf{[missing eqn ref]}}
\newcommand{\cond}{\mid}

\newcommand{\ab}{\bar{\a}} 
\newcommand{\z}{\zeta}
\newcommand{\zn}{\z_0}
\newcommand{\zz}{\boldsymbol{z}}
\newcommand{\ii}{\boldsymbol{i}}
\newcommand{\jj}{\boldsymbol{j}}
\newcommand{\uu}{\boldsymbol{u}}
\newcommand{\vv}{\boldsymbol{v}}
\newcommand{\zzz}{\boldsymbol{\zeta}}
\newcommand{\zzzargs}{\zzz(c,k,\a)}

\newcommand{\HH}{H_k(\a,\zzz;c)}
\newcommand{\HHone}{H_k(\a,\zzz;1)}
\newcommand{\D}{\Delta}
\renewcommand{\d}{\delta}
\newcommand{\ceil}[1]{\left\lceil {#1} \right\rceil}
\newcommand{\floor}[1]{\left\lfloor {#1} \right\rfloor}

\newcommand{\leb}{\leq O(1) \,}
\def\parenssq(#1){\left( {#1} \right)}
\newcommand{\parens}[1]{\left( {#1} \right)}
\newcommand{\abs}[1]{\left| {#1} \right|}

\newcommand{\ak}{\a_k}
\newcommand{\dak}{\ak/3}
\newcommand{\aak}{0.99 \ak}

\newcommand{\Var}{\operatorname{Var}}
\newcommand{\trans}{\mathsf{T}} 
\newcommand{\AAmn}{\mathcal{A}_{m,n}}
\newcommand{\BBmn}{\mathcal{B}_{m,n}}
\newcommand{\CCmn}{\mathcal{C}_{m,n}}
\newcommand{\Xmn}{X_{m,n}}
\newcommand{\Xmnl}{\Xmn^{(\ell)}}
\newcommand{\Ymn}{Y_{m,n}}
\newcommand{\Ymnl}{\Ymn^{(\ell)}}
\newcommand{\DD}{\mathcal{D}}

\newcommand{\gbsxx}[1]{\relax}
\newcommand{\MM}{\overline{M}}
\newcommand{\textfrac}[2]{{#1} \left/ {#2} \right.}
\newcommand{\mm}{{M}}
\newcommand{\nn}{{N}}
\newcommand{\mnhyp}{$m,n \to \infty$ with $\lim m/n \in (2/k,1)$}
\newcommand{\mnhypi}{$m,n \to \infty$ with $\lim m/n \in (2/k,\infty)$}
\newcommand{\mnanyhyp}{$m,n \to \infty$ with $\liminf m/n>2/k$}
\newcommand{\Pola}{X(\la)}
\newcommand{\Posla}{X_s(\la)}
\newcommand{\gftwo}{\ensuremath{\mathbb{F}_2}\xspace}
\newcommand{\Ortn}{O\Big(\frac1{\sqrt n}\Big)}
\newcommand{\alnu}{a(\ell,\nu)}
\newcommand{\bmnu}{b(m-\ell,\nu)}
\newcommand{\plnu}{p(\ell,\nu)}
\newcommand{\omegan}{w(n)}

\newcommand{\vt}{\vartheta}
\renewcommand{\t}{\tau}
\newcommand{\tup}{\boldsymbol{T}}
\newcommand{\tupt}{(\tup)_t} 
\newcommand{\lak}{\la_k} 
\newcommand{\tcdot}{\!\cdot\!} 
\newcommand{\gala}{g(\a;\la)}
\newcommand{\bigint}{[\aak,1/2]}
\newcommand{\awayint}{[\aak,0.49]}
\newcommand{\botint}{[\aak,\akstar]}

\newcommand{\Hk}{H_k} 
\newcommand{\event}{S} 
\newcommand{\akstar}{{\a^*_k}}
\newcommand{\astar}{\a^*}
\newcommand{\Hkala}{H_k(\a;\la)}
\newcommand{\Hka}{H_k(\a)}
\newcommand{\ahat}{\hat{\a}}
\newcommand{\lahat}{\hat{\la}}
\newcommand{\tmbf}[1]{\textbf{\boldmath{#1}}}
\newcommand{\mysubsubsec}[1]{\smallskip \noindent \tmbf{#1}.}
\newcommand{\psiinv}{\psi^{-1}}

\newcommand{\cm}{c^{-}}

\newcommand{\botintp}{[\aak,\akp]}
\newcommand{\topintp}{[\akp,1/2]}
\newcommand{\aplus}{\a^+}
\newcommand{\akp}{{\a^+_k}}
\newcommand{\akpt}{{0.4630}}
\newcommand{\Tnu}{T(\nu)}
\newcommand{\Tnux}[1]{T(#1)}
\newcommand{\nuf}{O(1/(\z_1 \sqrt \nu))}
\newcommand{\nf}{O(1/(\z_1 \sqrt n))}
\newcommand{\Hkhat}{\Hk(\ahat,\lahat)}
\newcommand{\betak}{0.001}
\newcommand{\Hap}{H(\akpt)}
\newcommand{\sharpconst}{0.69}
\newcommand{\Bin}{\operatorname{Bin}}

\title[Satisfiability Threshold for $k$-XORSAT]
{The Satisfiability Threshold for \lowercase{k}-XORSAT}

\author[Boris Pittel]{Boris Pittel}
\thanks{$^\dagger$
This research was supported by DIMACS, Center for Discrete
Mathematics and Theoretical Computer Science, Rutgers, the State University
of New Jersey, funded by the NSF under Grant No.\
DMS06-02942, Special Focus on Discrete Random Systems, and
by the NSF under Grants No.\ DMS-0805996, No.\ DMS 1101237.
The paper was begun when the second author was a researcher in
the Department of Mathematical Sciences,
IBM T.J.\ Watson Research Center, Yorktown Heights NY 10598, USA%
}
\address[Boris Pittel]{
Department of Mathematics \\
Ohio State University \\
Columbus OH 43210, USA}
\email{bgp@math.ohio-state.edu}

\author[Gregory B. Sorkin]{Gregory B. Sorkin$^\dagger$}
\address[Gregory B. Sorkin]{
Department of Management \\
London School of Economics and Political Science \\
Houghton Street \\
London WC2A 2AE
}
\email{g.b.sorkin@lse.ac.uk}

\allowdisplaybreaks

\begin{document}

\bibliographystyle{amsalpha}

\date{\today}

\begin{abstract}
We consider ``unconstrained'' random $k$-XORSAT,
which is a uniformly random
system of $m$ linear non-ho\-mo\-ge\-ne\-ous 
equations in \gftwo
over $n$ variables,
each equation containing $k\ge 3$ variables,
and also consider a ``constrained'' model where
every variable appears in at least two equations.
Dubois and Mandler proved that
$m/n=1$ is a sharp threshold for satisfiability of
constrained $3$-XORSAT,
and analyzed the 2-core of a random 3-uniform hypergraph to
extend this result to find the threshold 
for unconstrained 3-XORSAT.

We show that $m/n=1$
remains a sharp threshold for satisfiability of 
constrained $k$-XORSAT
for every $k\ge 3$,
and we use standard results on the 2-core of a random $k$-uniform hypergraph
to extend this result to find the threshold 
for unconstrained $k$-XORSAT.
For constrained $k$-XORSAT we narrow the phase transition window,
showing that 
$m-n \to -\infty$ implies almost-sure satisfiability, while
$m-n \to +\infty$ implies almost-sure unsatisfiability.
\end{abstract}

\maketitle

\allowdisplaybreaks

\section{Introduction}
An instance of $k$-XORSAT is given by a set of $m$
linear equations in \gftwo, over $n$ variables,
each equation involving $k$ variables 
and a right hand side which is either 0 or~1.
Equivalently, it is a linear system $Ax=b$ modulo 2
in which $A$ is an $m \by n$ 0--1 matrix
each of whose row sums is $k$,
and $b$ is an arbitrary 0--1 vector.

Random instances of many problems of this sort undergo phase transitions
around some critical ratio $\cstar$ of $m/n$,
meaning that for $m,n \to \infty$ with $\lim m/n < \cstar$,
the probability that a random instance $F_{n,m}$ is satisfiable
(or possesses some similar property) approaches $1$,
while if $\lim m/n > \cstar$
the probability approaches $0$.
(There is no loss of generality in hypothesizing the existence of a limit
since, in a broad context, a result as stated 
implies the same with the weaker hypotheses 
$\liminf m/n>\cstar$ and $\limsup m/n<\cstar$.)
Friedgut~\cite{Friedgut} proved that a wide range of problems
have such sharp thresholds,
but with the possibility that the threshold $\cstar = \cstar(n)$
does not tend to a constant.
The relatively few cases in which $\cstar$ is known to be a constant include
2-SAT, by
Chv\'atal and Reed~\cite{ChRe}, Goerdt~\cite{Goerdt},
and Fernandez de la Vega~\cite{FdlV}
(with the scaling window detailed by
Bollob\'as, Borgs, Chayes, Kim, and Wilson, \cite{BBCKW}),
an extension to Max 2-SAT, by
Coppersmith, Gamarnik, Hajiaghayi, and Sorkin~\cite{CGHS},
and the pure-literal threshold for a $k$-SAT formula, by Molloy~\cite{Molloy}.

The most natural random model of the $k$-XORSAT problem is
the ``unconstrained'' model
in which
each of the $m$ equations'
$k$ variables are drawn uniformly (without replacement)
from the set of all $n$ variables,
and the right hand side values are uniformly 0 or 1;
equivalently a random instance $Ax=b$ is given by
a matrix $A \in \bool^{m \by n}$ drawn uniformly at random
from the set of all such matrices with each row sum equal to $k$,
and $b \in \bool^m$ chosen uniformly at random.

The case $k=2$ has been extensively studied. 
As shown by Kolchin~\cite{Kolchin}
and Creignon and Daud\'e~\cite{CrDa}, 
the random instance has a solution with limiting probability $p(2m/n)+o(1)$, 
where $p(x)\in (0,1)$ for $x<1$, $p(1-)=0$, and $p(x)\equiv 0$ for $x>1$. 
Daud\'e and Ravelomanana~\cite{DaRa}, and Pittel and Yeum~\cite{PiYe},
analyzed the near-critical behavior
of the solvability probability for $2m/n=1+\varepsilon$, 
$\varepsilon =o(n^{-1/4})$.

For $k>2$, Kolchin~\cite{Kolchin}
analyzed the expected number of nonempty ``critical row sets''
(nonempty collections of rows whose sum is all-even),
whose presence is necessary and sufficient for the (Boolean) rank of $A$ 
to be less than $m$. 
He determined the thresholds $c_k$ such that the expected number
of nonempty critical sets goes to 0 if $\lim m/n<c_k$ 
and to infinity if $\lim m/n>c_k$;
in particular, $c_3=0.8894\dots$.
Thus, for $\lim m/n < c_k$, with high probability
$A$ is of full rank, so $Ax=b$ is solvable.
It follows
that the satisfiability threshold $\cstar_k$ is at least $c_k$.
It is an easy observation (see Remark~\ref{>1})
that $\cstar_k \leq 1$.
However, Kolchin could not resolve the precise value, 
or even the existence, of the satisfiability threshold.

Dubois and Mandler~\cite{DMfocs} (see also \cite{DM02})
introduced a ``constrained'' random $k$-XORSAT model,
where $b$ is still uniformly random, but
$A$ is uniformly random over the subset of matrices in which
each column sum is at least~2.
For $k=3$ (3-XORSAT) they showed that its threshold for $m/n$ is~1.
This is of interest because from the threshold for the constrained model,
they were able to derive that for the unconstrained model.
Dubois and Mandler 
suggested that their methods could be extended 
to the general constrained $k$-XORSAT, $k\ge 3$. 
However, their approach 
--- the second-moment method for the number of solutions --- 
requires solving a hard maximization problem with $\Theta(k)$ variables, 
a genuinely daunting task.

Our main result is that 1 continues to be the threshold for all $k>3$.
\begin{theorem} \label{main}
Let $Ax=b$ be a uniformly random constrained
$k$-XORSAT instance with $m$ equations and $n$ variables.
Suppose $k\ge 4$. 
If \mnhyp\ 
then $Ax=b$ is asymptotically almost surely (a.a.s.) satisfiable,
with satisfiability probability $1-O(m^{-(k-2)})$,
while if 
{$m,n \to \infty$ with $\lim m/n > 1$}
then $Ax=b$ is a.a.s.\ unsatisfiable,
with satisfiability probability $O(2^{-(m-n)})$.
\end{theorem}

We treat $k$ as fixed, and the constants implicit in the $O(\dot)$
notation may depend on $k$.
We are also able to treat the case
when the gap between $m$ and $n$ is not linear
but arbitrarily slowly growing,
obtaining the following stronger theorem.

\begin{theorem} \label{sharp}
Let $Ax=b$ be a uniformly random constrained
$k$-XORSAT instance with $m$ equations and $n$ variables,
with $k \geq 3$ and
\mnanyhyp. 
Then, for any $\omegan \to +\infty$,
if $m \leq n-\omegan$
then $Ax=b$ is a.a.s.\ satisfiable,
with satisfiability probability 
$1-O(m^{-(k-2)}+\exp(-\sharpconst \, \omegan))$,
while if $m \geq n+\omegan$
then $Ax=b$ is a.a.s.\ unsatisfiable,
with satisfiability probability $O(2^{-\omegan})$.
\end{theorem}

Rather than using the second-moment method on the number of solutions,
as Dubois and Mandler do,
we use the critical-set approach of Kolchin.
Remark~\ref{equiv} shows that the two methods are equivalent,
but Kolchin's leads us to more tractable calculations,
specifically, to a maximization problem with a number of variables
that is fixed, independent of $k$.
Using Kolchin's approach, but in the constrained model,
we will establish that $\cstar_k \geq 1$.
In the constrained and unconstrained models, 
a simple argument shows that $\cstar_k \leq 1$
(again see Remark~\ref{>1}).
Thus, for the constrained model (unlike the constrained one), 
the two bounds coincide, establishing the threshold.

Dubois and Mandler extended the threshold for 
the constrained 3-XORSAT model to
that for the unconstrained model
by observing that, in an unconstrained instance,
any variable appearing in just one clause (or none), 
can be deleted along with that clause (if any),
to give an equivalent instance, and this process can be repeated.
The key observation is that 
a uniformly random \emph{unconstrained} instance
reduces to
a uniformly random \emph{constrained} instance
with a predictable edge density;
the threshold for the unconstrained model is the value for which
the corresponding constrained instance has density 1. 
The same approach works for any $k$, 
and we capitalize on existing analyses of the 2-core of a 
random $k$-uniform hypergraph to 
establish the unconstrained $k$-XORSAT threshold
in Theorem~\ref{unconstrained}.

\subsection*{Other related work}
Work on the rank of random matrices over finite fields is not as extensive 
as that on real random matrices, but nonetheless a survey is beyond our scope.
In addition to the work already described, we note that
the rank of matrices with independent random 0--1 entries 
was explored over a decade ago by
Bl\"omer, Karp and Welzl \cite{BlKaWe},
and Cooper \cite{Cooper2}, among others.

In 2003, the $k$-XORSAT phase transition was determined by 
M{\'e}zard, Ricci-Tersenghi, and Zecchina \cite{Mezard}, 
by the non-rigorous ``replica'' method of statistical physics
and also by a second-moment calculation, 
with the purpose of showing that the replica method is correct in this
instance. 
The calculational details were omitted from the paper,
and the authors acknowledge \cite{MezardPersonal} that 
they did not rigorously prove negativity of the function playing the role
of our $H_k$ (see \eqref{H=})
nor did they treat the polynomial terms in the sum corresponding to our
\eqref{YlesseH} (which are of concern for small and large values of $\ell$).
Concurrently with and independently from our work, 
the k-XORSAT phase transition was also analyzed by
Dietzfelbinger, Goerdt, Mitzenmacher, Montanari, Pagh, and Rink
as part of a study of cuckoo hashing~\cite{DGMMPR10,DGMMPR09}.
We remark that where we work with critical row sets (see Section~\ref{Prback}),
counting vectors $y$ for which $yA=0$,
both M{\'e}zard et al.\ and Dietzfelbinger et al.\ use
the second-moment method on the number of solutions,
counting vectors $x$ for which $Ax=0$.

Recently, Darling, Penrose, Wade and Zabell \cite{DPWZ}
have explored a random XORSAT
model replacing the constant $k$ with a distribution,
but the satisfiability threshold has not yet been determined for
this generalization.

To translate our result for the constrained model to the unconstrained one,
we exploit results on the core of a random hypergraph.
For usual graphs, the threshold for the appearance of an $r$-core
was first obtained by 
Pittel, Spencer, and Wormald \cite{PiSpWo}.
For $k$-uniform hypergraphs, the $r$-core thresholds were obtained
roughly concurrently by
Cooper \cite{Cooper1},
Kim \cite{Kim}, and
Molloy \cite{Molloy}.
Two aspects of Cooper's treatment are noteworthy.
First, he works with a degree-sequence hypergraph model;
taking Poisson-distributed degrees reproduces the results for a simple random
hypergraph.
Also, he observes
\cite[Section 5.2]{Cooper1} that the
point at which a random $k$-uniform hypergraph's core 
has a (typical) edges-to-vertices ratio of 1
is an upper bound on the satisfiability threshold of 
unconstrained $k$-XORSAT;
proving that this is the true threshold is the 
main subject of the present paper.

\subsection*{Outline}
The remainder of the paper is organized as follows.
Section~\ref{Prback} formalizes our introductory observations about 
the first- and second-moment methods, 
the number of solutions, and the number of critical sets.
Section~\ref{PrSpaces} shows that
for the constrained model,
instead of considering random 0--1 matrices $A$,
it is asymptotically equivalent to consider random nonnegative 
\emph{integer} matrices $A$
subject to the same constraints on row sums (equal to $k$) 
and column sums (at least $2$).
Section~\ref{MainResult},
using generating functions and Chernoff's method, obtains an exponential bound
for the expected number of critical sets of any given cardinality.
Section~\ref{ProveLemma}
uses this bound
to show that, for $\lim m/n \in (2/k, 1)$ and $k>3$, 
the expected number of nonempty critical sets is $O(m^{-(k-2)})$.
Hence, with high probability,
there is no such set,
$A$ is of full rank, and the instance is satisfiable.
We conclude that 1 
is a sharp threshold for satisfiability of $Ax=b$ in
the constrained case for all $k\ge 3$.

Section~\ref{sharper} builds on the earlier results
to treat the case $\lim m/n=1$ and prove Theorem~\ref{sharp}.
Section~\ref{Unconstrained}
derives the unconstrained $k$-XORSAT threshold from the constrained one,
using standard results on the 2-core of a random hypergraph.

\section{Proof background}\label{Prback}

Let $N$ be the number of solutions 
to the system of equations $Ax=b$.

\begin{remark} \label{>1}
For an arbitrarily distributed $A \in \bool^{m \by n}$, 
with $b$ independent and uniformly distributed over $\bool^n$,
$\ex[N] = 2^{n-m}$,
and the satisfiability threshold is at most 1.
\end{remark}

\begin{proof}
Given $A$, there are $2^m$ systems given by $(A,b)$,
and in all they have $2^n$ solutions
since any $x$ uniquely determines $b=Ax$.
So $\E{N \cond A}=2^{n-m}$, and $\E{N}=2^{n-m}$.
By the first-moment method, 
$
\pr(Ax=b \text{ is satisfiable}) = \pr(N>0) \leq \ex[N] = 2^{n-m} ,
$
which tends to 0 if $\lim m/n >1$.
\end{proof}

\begin{definition}
Given a matrix, a \emph{critical set} is 
a collection of rows whose sum is all-even 
(i.e., the sum is the $0$ vector in $\gftwo$).
\end{definition}
\noindent
Note that the
collection of critical sets is sandwiched between
the minimal linearly dependent sets of rows, and 
all linearly dependent sets of rows.
It is useful because the minimal sets are hard to characterize,
while the collection of all linearly dependent row sets is
too large (as it includes all sets containing any linearly dependent sets);
the critical sets are a happy medium.

Let $X$ be the number of nonempty critical row subsets of a matrix $A$.
Where the first-moment method establishes the probable absence of solutions,
their probable presence can be established in this setting either by 
the second-moment method on the number of solutions,
showing that $\E{N^2} / \E{N}^2 \to 1$,
or by the first moment method on the number of non-empty critical row
sets, showing that $\E X \to 0$.
We will use the second approach (Kolchin's).
The two approaches suggest different calculations,
but as the following remark shows, they are equivalent.

\begin{remark} \label{equiv}
Let a distribution on $A \in \bool^{m \by n}$ be given, 
and let $b$ be independent of $A$ and uniformly distributed over $\bool^n$.
Then $\ex[N^2] \big/ \ex[N]^2 = \ex[X] +1$.
\end{remark}

\begin{proof}
Consider any fixed $A$, having rank $r(A)$ 
over $\gftwo$.
By elementary linear algebra, 
for each of the $2^{r(A)}$ values of $b$ in $\{Ax \colon x\in \{0,1\}^n\}$,
$Ax=b$ has $2^{n-r(A)}$ solutions,
giving $2^{2n-2\,r(A)}$ ordered pairs of solutions in each such case.
For the remaining values of $b$ there are no solutions,
so in all there are $2^{2n-r(A)}$ ordered pairs of solutions.
Taking the expectation over $b$ uniformly distributed over its $2^m$
possibilities, 
$\ex[N^2 \cond A] = \ex[2^{2n-r(A)-m}]$,
thus
$\ex[N^2] = \ex[2^{2n-r(A)-m}]$.
Since $\E{N}=2^{n-m}$ (see Remark~\ref{>1}),
\begin{equation*}
 \ex[N^2] / \ex[N]^2
  = \ex[2^{2n-r(A)-m}] / (2^{n-m})^2
  = \ex[2^{m-r(A)}] 
  = \ex[2^{n(A^\trans)}] ,
\end{equation*}
where $n(A^\trans)$ denotes the nullity of the transpose of $A$.

On the other hand, 
a critical row set is precisely one given by an indicator vector
$y \in \bool^m$ for which $y^\trans A=0$.
For a given $A$
the number of critical sets is thus $2^{n(A^\trans)}$,
and the expected number of non-empty critical row subsets is 
$\ex[X]=\E{2^{n(A^\trans)}}-1$.
\end{proof}

In fact, if $m \leq n$ and $\ex[X] \to 0$, then
with high probability $N=2^{n-m}$
(not merely $N/2^{n-m} \to 1$ in probability
as given by the second-moment method).
This follows because $X=0$ implies $r(A)=m$, in which case
$N=2^{n-m}$ for every $b$.
Thus,
$\Pr{N = 2^{n-m}}
 \geq \Pr{X=0}
 = 1- \Pr{X>0}
 \geq 1-\E X
 \to 1.
$

\si

The work in Sections \ref{PrSpaces}--\ref{ProveLemma} 
is to count the critical row subsets.
We will show that indeed $\ex[X]\to 0$ for the constrained random model
with $k\ge4$ and \mnhyp.

\section{Probability spaces}  \label{PrSpaces}
This section will establish Corollary~\ref{preq},
showing that the uniform distribution over constrained $k$-XORSAT matrices
$A \in \AAmn$ (see below)
is for our purposes equivalent
to a model $C \in \CCmn$ 
allowing a variable to 
appear more than once within an equation.

Let $\AAmn$ denote the set of all
$m\times n$ matrices with 0--1 entries, such that all $m$ row
sums are $k$, and all $n$ column sums are at least 2.
For $\AAmn$ to
be nonempty it is necessary that $km\ge 2n$,
and we will assume that 
\mnhyp.

A matrix $A\in \AAmn$ may be interpreted
as an outcome of the following allocation scheme. We have an
$m\times n$
array of cells with $k$ \emph{indistinguishable} chips assigned to each of
the $m$ rows. For each row, the $k$ chips are put in $k$ distinct cells
(so there is at most one chip per cell), subject
to the constraint that each column gets at least two chips.

Let us consider an alternative model, with the same constraints but
where the chips in each row are \emph{distinguishable},
giving allocations $B \in \BBmn$.
Then each allocation in $\AAmn$ is obtained
from $(k!)^m$ allocations in $\BBmn$,
and the uniform distribution on $\AAmn$ is equivalent to that on $\BBmn$.

Let $\CCmn$ be a relaxed version of $\BBmn$, without the requirement
that each of the $mn$ cells gets at most one chip.
Let $B$ and $C$ be
distributed uniformly on $\BBmn$ and $\CCmn$, respectively.
Crucially, and obviously, $B$ is equal in distribution to $C$, conditioned on $C\in
\BBmn$.

\smallskip

To state a key lemma on $\card\AAmn$, $\card\BBmn$, and
$\card\CCmn$ we need some notation, much of which will recur throughout the paper.

Introduce
\begin{align} \label{fdef}
f(x) &= \sum_{j\ge 2}\frac{x^j}{j!}=e^x-1-x 
,&
\psi(x) &= \frac{xf'(x)}{f(x)} ,  
\end{align}
with $\psi(0)=2$ defined by continuity,
and the \emph{truncated} Poisson random variable $Z=Z(\la)$,
\begin{align}         \label{Po2}
\pr(Z(\la)=j)=\frac{\la^j/j!}{f(\la)},\quad j\ge 2.
\end{align}

Then, 
\begin{align} \label{EZ}
\E{Z(\la)} &= \sum_j \frac{j \cdot \la^j/j!}{f(\la)} 
  = \frac{\la f'(\la)}{f(\la)} = \psi(\la) ,
\end{align}
Also,
$\E{Z(\la) (Z(\la)-1)} = \la^2 f''(\la)/f(\la)$,
leading to 
\begin{align} \label{VarZ}
\Var[Z(\la)] &= \E{Z (Z-1)} + \E{Z} - (\E{Z})^2
 \notag \\&=  \la^2 f''(\la)/f(\la) + \psi(\la) - (\psi(\la))^2
 = \la \psi'(\la) .
\end{align}
(With $\psi=x f'(x)/f(x)$,
\eqref{EZ} and \eqref{VarZ} hold for $f$ and $Z$ 
defined by any series $a_j x^j$, not just $x^j/j!$,
assuming convergence.)

From \eqref{VarZ} it is immediate that $\psi'(\la)>0$ for any $\la>0$.
(See also a general formulation in
\cite[Chapter 4, problem 6, p.~77]{Thompson}.)
The next claim shows that $\psi$ is convex as well as increasing,
and establishes both facts for all $\la$
(though we only require them for positive $\la$).

\begin{claim}\label{psi''}
$\psi(x)$ is strictly increasing, and convex. 
\end{claim}

\begin{proof}
We begin with convexity.
Differentiating \eqref{fdef} shows that 
$\psi'' = g/f^3$,
where
\begin{equation}
g=2f^2-4x f' f-x^2 f'' f+2x^2(f')^2 \label{gexplicit}
.
\end{equation}
Write $g(x) = \sum_{j\geq 0} g_j x^j$.
Expanding \eqref{gexplicit}
as a sum of terms $x^a e^{bx}=\sum_{j\ge 0}\tfrac{b^jx^{a+j}}{j!}$, 
and collecting like terms,
we find that $g_j=0$ for $j \leq 5$, while
for all $j \geq 6$,
$$
g_j=\frac{j-1}{j!}\bigl[2^{j-2}(j-8)+j^2-j+4\bigr]
 >0.
$$
Positivity is trivial for $j \geq 8$ and easily checked for $j=6$ and $7$.
This establishes that $g(x)>0$ for $x>0$.
Substituting $x=-y$ in $g(x)$,
writing $g(x) = e^{-2y} \sum_{j\ge 0}g'_j y^j$,
and using the same method yields
$g'_j=0$ for $j \leq 5$, while
for $j \geq 6$,
$g'_j = \frac1{j!} (2^{j+1}-j^3+4j^2-7j-4) >0$.
This establishes that $g(x)>0$ for $x<0$.
Finally, $\psi''(0)=1/9$. Therefore $\psi''(x)>0$ for all $x$.

That $\psi'(x)>0$ follows from 
$\lim_{x \to -\infty} \psi'(x)=0$ and $\psi''>0$.
\end{proof}

Under our assumption that $m/n>2/k$, 
the equation
$\psi(x) = km/n$
has a unique root, and it is positive.
This follows from the facts that
$\psi(x)$ is strictly increasing (see Claim~\ref{psi''}), 
 $\psi(0)=2$,
and $\psi(x) \to \infty$ as $x \to \infty$.
Henceforth, let 
\begin{align} \label{lambda}
\la &= \la(km/n)
 := \psiinv(km/n)
\end{align}
be this root. 
Since by Claim~\ref{psi''} $\psi$ is strictly increasing,
so is $\la=\psiinv$.

From \eqref{EZ} and \eqref{lambda},
\begin{align} \label{pgfE}
\E{Z(\la)} = \psi(\la) = \frac{km}{n}.
\end{align}
From Claim \ref{psi''}, for $\la>0$, $\psi'(\la)$
lies between $\psi'(0)=1/3$ and $\lim_{x \to \infty} \psi'(x)=1$,
and thus
\begin{align} \label{VarZ=}
\Var[Z(\la)]
&= \la \psi'(\la)
=\Theta(\la) .
\end{align}

\si

With these preliminaries done,
we focus on asymptotics of $\card\AAmn$, $\card\BBmn$ and $\card\CCmn$.

\begin{lemma}\label{LA}
Suppose \mnhypi. 
Then, with $\la$ as in \eqref{lambda}, 
\begin{align}
\card\CCmn=&\,\frac{1+O(n^{-1})}{\sqrt{2\pi n\Var[Z(\la)]}}\,(km)!\frac{f(\la)^n}{\la^{km}},
\label{C}\\
\frac{\card\BBmn}{\card\CCmn}=&\,\exp\left(-\frac{k-1}{2}\,\frac{\la e^{\la}}
{e^{\la}-1}\right)+o(1),\label{B}
\end{align}
so that the fraction $\card\BBmn/\card\CCmn$ is bounded away from zero.
Consequently
\begin{align}
\card\AAmn
= \frac{\card\BBmn}{(k!)^m}
=&\,\frac{1+o(1)}{\sqrt{2\pi n\Var[Z(\la)]}}\,\frac{(km)!}{(k!)^m}
\,\frac{f(\la)^n}{\la^{km}}
\exp\left(-\frac{k-1}{2}\,\frac{\la e^{\la}}{e^{\la}-1}\right).\label{A}
\end{align}
\end{lemma}
\bi

\begin{corollary} \label{preq}
Under the hypotheses of Lemma~\ref{LA},
uniformly for all non-negative, matrix-dependent functions~$r$,
\begin{equation*} 
\ex[r(A)]=
\ex[r(B)]=O(\ex[r(C)]).
\end{equation*}
\end{corollary}

\begin{proof}
The first equality is trivial.
To show the second,
for any $\event \subseteq \BBmn$,
\begin{align}   \label{P(B)=O(P(C))}
\pr(B\in \event )
 & =\pr(C\in \event \cond C\in \BBmn)
 \notag \\ &= \frac{\pr(C\in \event ,\,C\in \BBmn)}
{\card\BBmn \,/\,\card\CCmn }
\leq\,  \frac{\card\CCmn}{\card\BBmn} \,  {\pr(C\in\event )}
=O(1) \, \pr(C\in\event )
\end{align}
by \eqref{B}.
\end{proof}

\begin{proof}[Proof of Lemma~\ref{LA}]
Equation \eqref{A} is immediate from \eqref{C} and \eqref{B}.
Proving \eqref{C} and \eqref{B}
will occupy the rest of this section.

\textbf{We first prove \eqref{C}.} To determine $\card\CCmn$, recall that each row
$i \in m$ is
given its own $k$, mutually distinguishable, chips.
We can get an allocation $C \in \CCmn$ by permuting all the chips
and allocating the first $j_1 \geq 2$ chips to column~1,
the next $j_2 \geq 2$ chips to column~2, \textit{etc.};
each chip goes to its predetermined row and its random column.
Up to the irrelevant permutation of chips within the first $j_1$,
the next $j_2$, \textit{etc.},
an allocation $C \in \CCmn$ is uniquely determined by such a scheme.

Observe that the probability generating function (p.g.f.) of 
the truncated Poisson random variable $Z(\la)$ defined in \eqref{Po2} is 
\begin{align} \label{pgf}
\ex\bigl[z^{Z(\la)}\bigr]=\frac{f(z\la)}{f(\la)} .
\end{align}
Following the notational convention that for
$h(z)=\sum_jh_jz^j$, 
$ [z^j]\,h(z) := h_j $,
we have
\begin{align} 
\card\CCmn=&
 \,\sum_{j_1+\cdots+j_n=km\atop j_1,\dots,j_n\ge 2}
    \frac{(km)!}{j_1!\cdots j_n!}
 \notag \\
=&\,(km)!\,[z^{km}]\left(\sum_{j\ge 2}\frac{z^j}{j!}\right)^n
 =(km)!\,[z^{km}]f(z)^n
 \label{14.1}
 \\
=&\,(km)!\frac{f(\la)^n}{\la^{km}}\,[z^{km}]\left(\frac{f(z\la)}{f(\la)}\right)^n
 \notag \\
=&(km)!\frac{f(\la)^n}{\la^{km}}\,[z^{km}]\bigl(\ex[z^{Z(\la)}]\bigr)^n
 \quad \text{(see \eqref{pgf})}
 \notag \\
=&\,(km)!\frac{f(\la)^n}{\la^{km}}\,\pr\left(\sum_{j=1}^nZ_j(\la)=km\right),
 \label{|Cmn|}
\end{align}
where $Z_1(\la),\dots,Z_n(\la)$ are independent copies of $Z(\la)$.
Now, since
$\Var[Z(\la)]=\Theta(\la)$ 
(by \eqref{VarZ=}) and $\liminf\la>0$ (by $\la=\la(km/n)$ and the hypothesis
that $\lim m/n > 2/k$), 
we have
$\liminf\Var[Z(\la)]>0$.
So, by a local limit theorem
(Aronson, Frieze and Pittel~\cite[equation (5)]{ArFrPi}), 
$$
\pr\left(\sum_{j=1}^nZ_j(\la)=km\right)=\pr\left(\sum_{j=1}^nZ_j(\la)=n\ex[Z(\la)]\right)=
\frac{1+O(n^{-1})}{\sqrt{2\pi n\mathrm{Var}[Z(\la)]}},
$$
which proves \eqref{C}.

\textbf{We now prove \eqref{B}.} Let $C=\{c_{i,j}\}$ be distributed uniformly on $\CCmn$. Let $M$ denote
the number of cells that house $2$ or more chips, i.e.,
$
M=\bigl|\{(i,j) \colon c_{i,j}\ge 2\}\bigr|.
$
Let
$\MM$ be the number of pairs of chips hosted by the same cell, i.e.,
$$
\MM=\sum_{(i,j) \colon c_{i,j}\ge 2}\binom{c_{i,j}}{2}=\sum_{(i,j)}\binom{c_{i,j}}{2}.
$$
$\MM=M$ iff there are no cells hosting more than $2$ chips.
Clearly
$$
\frac{\card\BBmn}{\card\CCmn}=\pr(C\in\BBmn)=\pr(\MM=0).
$$
Of course, $\pr(\MM=0)=\pr(M=0)$, but, unlike $M$, $\MM$ is amenable to moment
calculations.
\si

Denoting the indicator of an event $E$ by $\one(E)$, we write
\begin{align}
\MM &=\sum_{i\in [m],\,j\in [n]}\sum_{1\le u<v\le k}\one\bigl(E(i,j;u,v)\bigr),
 \label{MM}
\end{align}
where $E(i,j;u,v)$ is the event that, of the $k$ chips owned by row $i$, at
least the two chips
$u$ and $v$ were put into cell $(i,j)$. Each of these $mn\binom{k}{2}$ event
indicators has the same expected value,
\begin{equation}\label{Exp}
\ex[\one(E(i,j;u,v))] =
(km-2)!\,\frac{[x^{km-2}]f(x)^{n-1}e^x}{\card\CCmn} .
\end{equation}
To see why \eqref{Exp} is so, compare with \eqref{14.1}
and note that once we have put two selected chips into a cell $(i,j)$
we allocate the remaining $(km-2)$ chips
amongst $n$ columns, at least two per column,
with the exception (hence the sole $e^x$ factor)
that the $j$th column receives an unconstrained
number of additional chips (as it already has two).
Arguing as for \eqref{|Cmn|},
\begin{align}\label{Arg}
[x^{km-2}]f(x)^{n-1}e^x
 &=
 \frac{f(\la)^{n-1}e^{\la}}{\la^{km-2}} \; 
  \Pr{\sum_{j=1}^{n-1}Z_j(\la)+ \Pola=km-2} ,
\end{align}
where $\Pola$ stands for an independent, 
usual (not truncated) Poisson$(\la)$ random variable. 
This last probability equals
\begin{align*}
 \sum_r \; \Big[
  \Pr{ \Po(\l)=r } \: \cdot \:
  \pr \Big( \sum_{j=1}^{n-1} Z_j(\l)= km-2-r \Big)
 \Big] .
\end{align*}
By the local limit theorem for $\sum_{j=1}^{n-1} Z_j(\l)$,
for $r \leq \ln n$
the second probability in the $r$th term of the sum is again
asymptotic to $(2\pi n \Var[Z(\l)])^{-1/2}$.
Then so is the probability in \eqref{Arg},
since $\Pr{\Pola > \ln n} = O(n^{-K})$, for every $K>0$.
From this, \eqref{MM}, \eqref{Exp}, \eqref{Arg}, and \eqref{C}, 
$$
\ex[\MM] 
 =(1+o(1))\frac{mn\binom{k}{2}}{(km)_2}\,\frac{\la^2e^{\la}}{f(\la)}
$$
with the usual falling-factorial notation
$(a)_b:=a(a-1)\cdots(a-b+1)$.  
Recalling \eqref{lambda} 
and setting
\begin{equation} \label{gammaDef}
\gamma := \frac{k-1}{2}\,\frac{\la e^{\la}}{e^{\la}-1}
\end{equation}
gives
$$
\ex[\MM] 
 =\gamma+o(1).
$$

More generally, we now show that for every fixed $t\ge 1$ we have
\begin{equation}\label{ENt}
 \ex[(\MM)_t] = \gamma^t +o(1).
\end{equation}
Let $\tup$ be the set of all 4-tuples $(i,j,u,v)$ as before,
with $i \in [m]$, $j \in [n]$, and $1 \leq u < v \leq k$.
Now let $\tupt$ denote the collection of $t$-tuples
of such 4-tuples with all the 4-tuples distinct.
Where $\set{(i_\t,j_\t,u_\t,v_\t)}_{\t=1}^t \in \tupt$,
in a slight abuse of notation we will write $(\ii,\jj,\uu,\vv) \in \tupt$,
where
$\ii=(i_1,\dots,i_t)$, $\jj=(j_1,\dots,j_t)$,
$\uu=(u_1,\dots u_t)$, $\vv=(v_1,\dots,v_t)$.
Then we have
\begin{align*}
(\MM)_t &=
\sum_{(\ii,\jj,\uu,\vv) \in \tupt}
\one\left(\bigcap_{s=1}^t E(i_t,j_t;u_t,v_t)\right) ,
\intertext{hence}
\ex[(\MM)_t] &=
\sum_{(\ii,\jj,\uu,\vv) \in \tupt}
\pr\left(\bigcap_{s=1}^t E(i_t,j_t;u_t,v_t)\right).
\end{align*}
We break the sum into two parts, $\Sigma_1$ and the remainder $\Sigma_2$, where
$\Sigma_1$ is the restriction to
$\ii$ and $\jj$ each having all its components distinct. In
$\Sigma_1$ the number of summands is $(m)_t(n)_t\binom{k}{2}^t$, and each
summand is
$$
(km-2t)!\,\frac{[x^{km-2t}]\,f(x)^{n-t}(e^x)^t}{\card\CCmn};
$$
see the explanation following \eqref{Exp}. Analogously to \eqref{Arg},
\begin{equation*}
 [x^{km-2t}]f(x)^{n-t}(e^x)^t
 =\frac{f(\la)^{n-t}(e^{\la})^t}{\la^{km-2t}} 
  \, \Pr{ \; \sum_{j=1}^{n-t}Z_j(\la)+\sum_{s=1}^t \Posla = km-2t},
\end{equation*}
where the $n$ truncated and ordinary Poisson random variables
$Z_j(\la)$ and $\Posla$ are mutually independent.
Since $t$ is fixed, the probability remains
asymptotic to $\bigl(2\pi n\Var[Z(\la)]\bigr)^{-1/2}$. So, using \eqref{C} and
recalling \eqref{gammaDef}, we have
\begin{align}
\Sigma_1\sim&\,\frac{(m)_t(n)_t\binom{k}{2}^t}{(km)_{2t}}\,\left(\frac{\la^2e^{\la}}{f(\la)}
\right)^t\notag\\
\sim&\,\left[\frac{mn\binom{k}{2}}{(km)^2}\,\frac{\la^2e^{\la}}{f(\la)}\right]^t\to\gamma^t.
\label{Sig1}
\end{align}
In the case of $\Sigma_2$, 
letting $I=\{i_1,\dots,i_t\}$, $J=\{j_1,\dots,j_t\}$, we have
$|I|+|J|\le 2t-1$. So the number of attendant pairs $(I,J)$ is at most $(m+n)^{2t-1}
=O\bigl(m^{2t-1}\bigr)$. 
The number of pairs $(\ii,\jj)$ inducing a given pair $(I,J)$ is bounded above by a constant $s(t)$. 
For every one of those $s(t)$ choices, we select pairs of chips for each of
the chosen $t$ cells; there are at most $\binom{k}{2}^t$ ways of doing so.
Lastly, we allocate the remaining $(km-2t)$ chips in such a way that every
column $j\in [n]\setminus J$ gets at least $2$ chips. 
As in the case of $\Sigma_1$, this can be done in
\begin{multline*}
(km-2t)!\,[x^{km-2t}]\,f(x)^{n-|J|}(e^x)^{|J|}\\
=
(km-2t)!\,
\frac{f(\la)^{n-|J|}(e^{\la})^{|J|}}{\la^{km-2t}}\,\,\pr\!\!\left(\sum_{j=1}^{n-|J|}Z_j(\la)+
\sum_{s=1}^{|J|}\Po_s(\la)=km-2t\right)
\end{multline*}
ways. 
Again, the probability is asymptotic to $\bigl(2\pi n\var[Z(\la)]\bigr)^{-1/2}$. So,
as $e^{\la}>f(\la)$, the sum $\Sigma_2$ is of order
\begin{equation}\label{Sig2}
m^{2t-1}\,\frac{(km-2t)!}{(km)!}\left(\frac{e^{\la}\la^2}{f(\la)}\right)^{t}=O(m^{2t-1}/m^{2t})=O(m^{-1}).
\end{equation}
Combining \eqref{Sig1} and \eqref{Sig2}, and recalling \eqref{gammaDef}, we
conclude that for each fixed $t\ge 1$,
$$
\ex[(\MM)_t] = \gamma^t+o(1).
$$
Therefore $\MM$ is asymptotic, with all its moments and in distribution, to $\Po(\gamma)$.
In particular,
$$
\pr(\MM=0) =\pr(\Po(\gamma)=0)+o(1)=e^{-\gamma}+o(1).
$$
This completes the proof of Lemma~\ref{LA}.
\end{proof}
\bi

\section{Counting critical row subsets, and the main result}
\label{MainResult}

This section will prove Theorem~\ref{main}.
Remark~\ref{>1} already dealt with the case $\lim m/n>1$.
It suffices, then, to show that with $\lim m/n \in (2/k,1)$,
the expected number of nonempty critical row sets goes to 0:
then with high probability there is no such set,
$A$ is of full rank, and the instance is satisfiable.

In the model $\CCmn$,
Lemma~\ref{expbounds} gives an upper bound on
the expected number of critical row sets of each cardinality 
$\ell \in \set{1,\ldots,m}$ 
as a function of $c=m/n$, $k$, $n$, and $\ell$, 
minimized over two additional variables $\z_1$ and $\z_2$.
Lemma~\ref{Hlemma} shows that, for $c \in (2/k,1)$,
there exist values for $\z_1$ and $\z_2$ 
making this bound small,
in particular making its exponential dependence on $n$
decreasing rather than increasing.
Corollary~\ref{dupCor} uses Lemma~\ref{Hlemma} to show that
in the model $\AAmn$
the total expected number of nonempty critical row sets is of order
$O\bigl(m^{-(k-2)}\bigr)$,
proving Theorem~\ref{main}.

Lemma~\ref{Hlemma} is established by several claims deferred to 
Section~\ref{ProveLemma},
and Section~\ref{Unconstrained} extends Theorem~\ref{main} to the 
unconstrained $k$-XORSAT model (Theorem~\ref{unconstrained}).


\begin{lemma}\label{expbounds} 
Suppose $k\ge 3$ and \mnhypi,
and let $C$ be chosen uniformly at random 
from $\CCmn$.
For $\ell \in \set{1,\ldots,m}$,
let $\Ymnl$ denote the number of 
critical row sets of $C$ of cardinality~$\ell$.
Then, with $c=m/n$,
$\a=\ell/m$,
$\ab=1-\a$,
$\l=\l(ck)$ as given by \eqref{lambda},
and $\zzz=(\z_1,\z_2)>\zero$,
\begin{equation}\label{YlesseH}
\ex\bigl[\Ymnl\bigr]
  \leb \,
 \sqrt{\tfrac{1}{\z_2}}\,
 \exp\bigl[n H_k(\a,\zzz;c)\bigr],
\quad\forall\,\zzz>0,
\end{equation}
where
\begin{multline}\label{H=}
H_k(\a,\zzz;c)=c H(\a)
+ck\a \ln(\a/\z_1) +ck\ab \ln(\ab/\z_2)
\\* 
 +\ln\frac{f(\l\tcdot(\z_2+\z_1))+f(\l\tcdot(\z_2-\z_1))}{2f(\la)} ,
\end{multline}
by continuity we define $x \ln x=0$ at $x=0$,
and $H(\a)$ is the usual entropy function
$$
H(\a):=
-\a\ln\a-(1-\a)\ln({1-\a}) .
$$
\end{lemma}

\begin{proof}
By symmetry,
\begin{equation}\label{Ymnl}
\ex[\Ymnl]=\binom{m}{\ell}\pr(\DD_{\ell});\quad
\DD_{\ell}:=\bigcap_{j=1}^{n}\left\{\sum_{i=1}^{\ell}c_{i,j}
\text{ is even}\right\}.
\end{equation}
By symmetry again,
\begin{align}\label{nuth}
\pr(\DD_{\ell})
=&\,\sum_{\nu=1}^n\binom{n}{\nu} \pr(\DD_{\ell,\nu}), 
\end{align}
where
\begin{align}
 \DD_{\ell,\nu}:=&\,
\bigcap_{j=1}^{\nu}\left\{\sum_{i=1}^{\ell}c_{i,j}
\text{ is even, positive}\right\}\bigcap\bigcap_{j=\nu+1}^n\left\{\sum_{i=1}^{\ell}c_{i,j}=0\right\}.
\end{align}
Recalling that $\sum_{i\in [m]}c_{i,j}\ge 2$, 
we see that on the event $\DD_{\ell,\nu}$,
\begin{equation}\label{sepcon}
\sum_{i\le\ell}c_{i,j}=\left\{\begin{alignedat}2
&\text{even }>0,\quad&&j\le\nu,\\
&0,\quad&&j>\nu;\end{alignedat}\right.\qquad
\sum_{i>\ell}c_{i,j}\geq\left\{\begin{alignedat}2
&0,\quad&&j\le\nu,\\
&2,\quad&&j>\nu.\end{alignedat}\right.
\end{equation}
Thus on $\DD_{\ell,\nu}$ the column sums of the two complementary
submatrices, $\{c_{i,j}\}_{i\leq \ell,j\in [n]}$ and $\{c_{i,j}\}_{i>\ell,j\in [n]}$, are subject
to independent constraints.

Let $\CCmn(\ell,\nu)$ denote the set of all matrices $C$ with
row sums $k$ which meet the constraints \eqref{sepcon}. 
Then $\pr(\DD_{\ell,\nu})$ is given by
\begin{equation}\label{plnuless}
 \plnu :=\pr(\DD_{\ell,\nu})=\frac{|\CCmn(\ell,\nu)|}{\card\CCmn}.
\end{equation}
By the independence of constraints on column sums for the upper and the lower
submatrices of the matrices $C$ in question,
\begin{equation} \label{c=ab}
|\CCmn(\ell,\nu)|=\alnu \cdot \bmnu,
\end{equation}
where (paralleling our definition of $\CCmn$ in Section~\ref{PrSpaces})
$\alnu $ is the number of ways to assign
$k\ell$ chips among the first $\nu$ columns so that each of those columns gets a positive even number of chips,
and 
$\bmnu$ is the number of ways to assign $k(m-\ell)$ chips among
all $n$ columns so that each of the last $(n-\nu)$ columns gets at least $2$ chips.

As in \eqref{|Cmn|},
\begin{align}
\alnu 
&=
\sum_{j_1+\cdots+j_{\nu}=k\ell\atop j_s>0,\text{ even}}
\frac{(k\ell)!}{j_1!\cdots j_{\nu}!}
\notag \\ &= 
(k\ell)!\,[z^{k\ell}]\left(\sum_{j>0,\text{ even}}\frac{z^j}{j!}\right)^{\nu}
\notag \\&
=(k\ell)!\,[z^{k\ell}](\cosh z -1)^{\nu},  \label{aseries}
\end{align}
and
\begin{align}
\bmnu=&\,\sum_{j_1+\cdots+j_n=k(m-\ell)\atop
j_1,\dots,j_{\nu}\ge 0;\,\, j_{\nu+1},\dots,j_n\ge 2}
\frac{(k(m-\ell))!}{j_1!\cdots j_n!}
\notag \\ =&
(k(m-\ell))!\,[z^{k(m-\ell)}] (e^z)^{\nu} f(z)^{n-\nu}.   \label{bseries}
\end{align}
Since the coefficients of the Taylor expansion around $z=0$ of 
$e^{z\nu}f(z)^{n-\nu}$
are non-negative, we use these identities in a standard (Chernoff) way to bound
\begin{equation}\label{abless}
\begin{aligned}
\alnu \le&\,(k\ell)!\,\frac{(\cosh z_1 -1)^{\nu}}{z_1^{k\ell}},\quad
\forall\, z_1>0 .
\end{aligned}
\end{equation}
We could control $\bmnu$ similarly, but
we need a stronger bound, namely
\begin{equation}\label{blessbetter}
\bmnu\leb (nz_2)^{-1/2}(k(m-\ell))!\,\frac{(e^{z_2})^{\nu}f(z_2)^{n-\nu}}{z_2^{k(m-\ell)}},
\quad\forall\,z_2>0.
\end{equation}
The bound \eqref{blessbetter}
follows from three components:
the Cauchy integral formula
$$
\bmnu
 =
 \frac{(k(m-\ell))!}{2\pi}\!\!
 \oint\limits_{{\mbox{\Large${z=z_2e^{i\theta} \colon 
     \atop\theta\in (-\pi,\pi]}$}}\!\!}
 \frac {(e^{z})^{\nu}f(z)^{n-\nu}} {z^{k(m-\ell) +1 }} \, dz,
$$
and (with $z=z_2 e^{i\theta}$) 
the identity
$|e^z|=e^{z_2}\exp\bigl[-z_2(1-\cos\theta)\bigr]$ 
and the less obvious inequality
\begin{align} 
 |f(z)| & \le |f(z_2)|\exp\bigl[-z_2(1-\cos\theta)/3\bigr].
 \label{PittelIneq}
\end{align} 
(See Pittel~\cite[Appendix]{Pittel86} for the inequality,
and Aronson, Frieze and Pittel~\cite[inequality (A2)]{ArFrPi} 
for how it works in combination with the Cauchy formula.)

Using 
\eqref{plnuless}, 
\eqref{c=ab}, \eqref{abless}, \eqref{blessbetter},
with $\card\CCmn$ from
\eqref{C} 
and $\Var{Z(\la)}$ from \eqref{VarZ=}, we obtain
that, $\forall\,z_1,z_2>0$,
\begin{equation}
 \plnu \leb \sqrt{\frac{\la}{z_2}}\,\binom{km}{k\ell}^{-1}\frac{\la^{km}}{z_1^{k\ell}\,z_2^{k(m-\ell)}}\,\,
\frac{[e^{z_2}(\cosh z_1-1)]^{\nu}f(z_2)^{n-\nu}}{f(\la)^n} .
\label{peq}
\end{equation}

Now, it is immediate from
\eqref{Ymnl}, \eqref{nuth}, and \eqref{plnuless} that
\begin{align}
\ex\bigl[\Ymnl\bigr]
 = \binom m \ell \sum_{\nu=1}^n \binom n \nu  \plnu  .
 \label{EYmn1}
\end{align}
If we restrict to $z_1$ and $z_2$ depending only on $\ell$, $m$ and $n$
(not on $\nu$),
then on substituting \eqref{peq} into \eqref{EYmn1}
we may simplify the sum using the binomial formula to obtain
\begin{align}
\ex\bigl[\Ymnl\bigr]
 \leb 
 &\, \sqrt{\frac{\la}{z_2}}\, \binom{m}{\ell}\binom{km}{k\ell}^{-1} \la^{km}
  \notag\\*
 \times&\,\frac{1}{z_1^{k\ell}\,z_2^{k(m-\ell)}}
 \parens{ \frac{f(z_2)+e^{z_2}(\cosh z_1-1)}{f(\la)}}^n
 ,\quad\forall\,z_1,z_2>0 \label{Ymnlless}
.
\end{align}
Observe that
$$
f(z_2)+e^{z_2}(\cosh z_1-1)=\frac{f(z_1+z_2)+f(z_2-z_1)}{2}.
$$
Inequality \eqref{YlesseH}, and thus the lemma,
are established by substituting this
and the Stirling-based approximation 
$\binom n {pn}=O(1) \frac1{\sqrt{ {n p(1-p)}}} \exp(n H(p))$
into \eqref{Ymnlless}, recalling that
$m=cn$, 
$\a=\ell/m$ and $\ab=1-\a$,
substituting $z_1 = \z_1 \l$ and $z_2 = \z_2 \l$, 
and observing that $\sqrt k = O(1)$.
For $\ell=m$ the Stirling-based approximation is inapplicable
but consistency of \eqref{YlesseH} with \eqref{Ymnlless}
is easily checked.
\end{proof}

Recall the definition of $\HH$ from \eqref{H=}.
Roughly speaking,
the following lemma
establishes the existence of $\zzz$ making $\HH$ negative.
An intuitive description of the behavior of 
$\HH$ is given at the start of the next section.

\begin{lemma} \label{Hlemma}
Let
\begin{align}
\ak &= e k^{-k/(k-2)} .   \label{ak}
\end{align}
For all 
$k \geq 4$
and $c \in (2/k,1)$,
there exist 
$\e=\e(c,k)>0$ and $\zn=\zn(c,k)>0$, both functions continuous in $c$,
such that
\begin{align}
 \big( \forall \a \in (0,\ak] \, \big) \, (\exists \zzz)
  & \colon
 \HH \leq (c \a)(\tfrac k 2-1) \ln ( \a / \ak)
  \text{ and } \z_2 \geq \zn  \hspace*{-2cm}
  \label{lemma1.2}
\\
 \big( \forall \a \in [\dak,1] \, \big) \, (\exists \zzz)
  & \colon
  \HH \leq -\e \text{ and } \z_2 \geq \zn .
  \label{lemma1.1}
\end{align}
\end{lemma}

\begin{proof}
The lemma follows immediately from 
Claims~\ref{asmall}, 
\ref{alarge} and~\ref{ahalf},
all stated and proved in Section~\ref{ProveLemma},
respectively treating $\a$ in the ranges 
$(0, \aak]$,
$[\aak,1/2]$,
and
$(1/2,1]$.
A suitable function $\zzz$ is given explicitly in each case.
\end{proof}

The lemma yields the following corollary.
\begin{corollary} \label{dupCor}
For $k \geq 4$ and $m,n \to \infty$ with $\lim m/n \in (2/k,1)$,
$$
\sum_{\ell=2}^{m}\ex\bigl[\Ymnl\bigr]
=O\bigl(m^{-(k-2)}\bigr).
$$
\end{corollary}

\begin{proof}
Since $\lim m/n \in (2/k,1)$, there exists 
a closed interval $I \subset (2/k,1)$ 
such that,
for all but finitely many cases,
$c=m/n \in I$. 
Where $\e(c,k)$ and $\zn(c,k)$ satisfy the conditions of Lemma~\ref{Hlemma}
define $\e = \e(I) = \min\set{\e(c,k) \colon c \in I}$,
and $\zn = \zn(I)$ likewise;
the minima exist by continuity of $\e$ and $\zn$ in $c$.
Then, for all but finitely many pairs $m,n$, 
inequalities \eqref{lemma1.2} and \eqref{lemma1.1} hold true.

Letting $\ell_k = \ak m = \Theta(n)$,
for $\ell \leq \ell_k/2$,
recalling that $\a c n = \a m = \ell$,
\eqref{YlesseH} and \eqref{lemma1.2} give
\begin{align} \label{ymnlBound}
\ex\bigl[\Ymnl\bigr]
 =O(1) \,
 \exp\bigl[(\tfrac k2-1) \ell \ln(\ell/\ell_k)
  \bigr] ,
\end{align}
where we have incorporated $\sqrt{\textfrac1 \zn}$ in the leading $O(1)$.
By convexity of
$\ell \ln(\ell/\ell_k)$,
interpolating
for $\ell \in [2, \ell_k/2]$
from the endpoints of this interval,
\begin{align}
\ell \ln(\ell/\ell_k)
 & \leq
 2 \ln(2/\ell_k)
  + \frac{\ell-2}{\ell_k/2-2} \parens{(\ell_k/2) \ln(1/2) - 2 \ln(2/\ell_k)}
  \notag
 \\ &=
 2 \ln(2/\ell_k) + (\ell-2) (-\ln 2 + o(1)) ,
 \notag
 \\ & \leq 2 \ln (2/\ell_k) - 0.6(\ell-2)
 \notag
\end{align}
for $n$ sufficiently large, where we have used that $\ell_k = \Theta(n)$
and $0.6 < \ln 2$.
Thus, 
for $\ell \in [2, \ell_k/2]$,
\begin{align*}
\ex\bigl[\Ymnl\bigr]
 & \leq
 O(1) \, \exp \big( (\tfrac k2-1) [2 \ln (2/\ell_k) -0.6 (\ell-2) ] \big)
 \\ & =
 O(1) \, m^{-(k-2)} \exp(-0.6 (\tfrac k2-1) (\ell-2)) ,
\end{align*}
where the last line 
incorporates $(2/\ak)^{k-2}$
in the $O(1)$.
Given this upper bound that is geometrically decreasing in $\ell$,
summing gives
$$
 \sum_{\ell=2}^{\floor{(\ak/2) m}}
  \ex\bigl[\Ymnl\bigr]
  = O \bigl(m^{-(k-2)}\bigr) .
$$

For $\ell > (\ak/2) m$, by \eqref{lemma1.1},
$ \ex\bigl[\Ymnl\bigr]
 =O(1) \,
 \exp(-\e n),
$
giving
$$\sum_{\ell=\ceil{(\ak/2) m}}^{m}
\ex\bigl[\Ymnl\bigr]
  = O(m) \exp(-\e n) 
  = \exp(-\Omega(n)) .
$$
Adding the two partial sums yields Corollary~\ref{dupCor}.
\end{proof}

\begin{proof}[Proof of Theorem~\ref{main}]
By the remarks at the start of this section, 
we need only consider the case $\lim m/n \in (2/k,1)$.
Under the hypotheses of Corollary~\ref{dupCor},
let $A \in \AAmn$ and $C \in \CCmn$ be uniformly random,
and let $\Xmn$ and $\Ymn$ denote the numbers of nonempty critical row sets of 
$A$ and $C$ respectively,
and $\Xmnl$ and $\Ymnl$ those of cardinality $\ell$.
$\Xmn^{(1)}=0$ since every row of $A$ has $k$ 1's.
(The bound on $\Ymn^{(1)}$ from \eqref{ymnlBound} is
$O(m^{-(\tfrac k2-1)})$, whose use would weaken the Corollary's conclusion.
$\Ymn^{(1)}$ is not necessarily $0$ since a row of $C$ can be 0,
for example if all the 1's in its defining configuration lie in a single cell.)
Then
$$
\ex\bigl[\Xmn\bigr]
= 0+ \sum_{\ell=2}^{m}\ex\bigl[\Xmnl\bigr]
= O(1) \sum_{\ell=2}^{m}\ex\bigl[\Ymnl\bigr]
=O\bigl(m^{-(k-2)}\bigr),
$$
the last two equalities coming from Corollary~\ref{dupCor}
and Corollary~\ref{preq}.
Then $\Pr{A \text{ is not of full rank}} \leq \E{\Xmn} = O(m^{-(k-2)})$,
so with probability $1-O(m^{-(k-2)})$,
$A$ is of full rank and any system $Ax=b$ is satisfiable.
\end{proof}

\section{\texorpdfstring%
{Analysis of the function ${\HH}$ to prove Lemma~\ref{Hlemma}}%
{Analysis of the function H to prove Lemma~\ref{Hlemma}}} 
\label{ProveLemma}
Recall the notation $\ab=1-\a$ and $\zzz=(\z_1,\z_2)$ 
as well as the definition of $\HH$ from \eqref{H=}.
In this section we 
use an explicit function $\zzz=\zzzargs$,
taking different forms in different ranges of $\a$,
to establish Claims~\ref{asmall}, 
\ref{alarge} and~\ref{ahalf}
and thus Lemma~\ref{Hlemma}.

For intuition about $\HH$, the case $k=4$ is indicative.
Figure~\ref{H4} shows a graph of the function value against $\a$, 
for a few choices of $c$,
with $\zzz$ given by \eqref{zsmall} for small $\a$, 
and by $\zzz=(\a,\ab)$ otherwise.
Numerical experiments suggest that the optimal choice of $\zzz$
leads to qualitatively similar results, though of course without the
kinks where we change from one functional form for $\zzz$ to another.
As shown, $\HH$ tends to 0 at $\a=0$ (treated in Claim~\ref{asmall}), 
but the dependence on $c$ here is not critical:
an analog of the claim, with different parameters,
could be obtained as long as $c$ is bounded away from 0 and infinity.
At $\a=1/2$ (treated in Claim~\ref{alarge}), 
the function tends to 0 as $c$ tends to 1,
so this is where $c < 1$ is required.
Claim~\ref{alarge} also covers
values of $\a$ between 0 and $1/2$ but bounded away from them;
here the function value is bounded away from 0 (for $c \leq 1$)
and could be dealt with by cruder means,
such as that by interval arithmetic in \cite{ArxivGreg}.
Function values for $\a>1/2$ 
(treated in Claim~\ref{ahalf})
are dominated
by their symmetric counterparts at $1-\a$,
except for some special treatment required near~$1$.

Lemma~\ref{Hlemma} only considers $k>3$.
The lemma can be extended to $k=3$, but this case
was already treated by \cite{DMfocs}
and the proof poses additional difficulties for us;
see further discussion after the proof of Claim~\ref{alarge},
and in Section~\ref{sharper}, specifically at \eqref{barzs}.

\begin{figure}[htbp]
\begin{centering}
\ifpdf{%
\includegraphics[bb=0.0in 0.0in 10in 7.5in,width=4.0in]{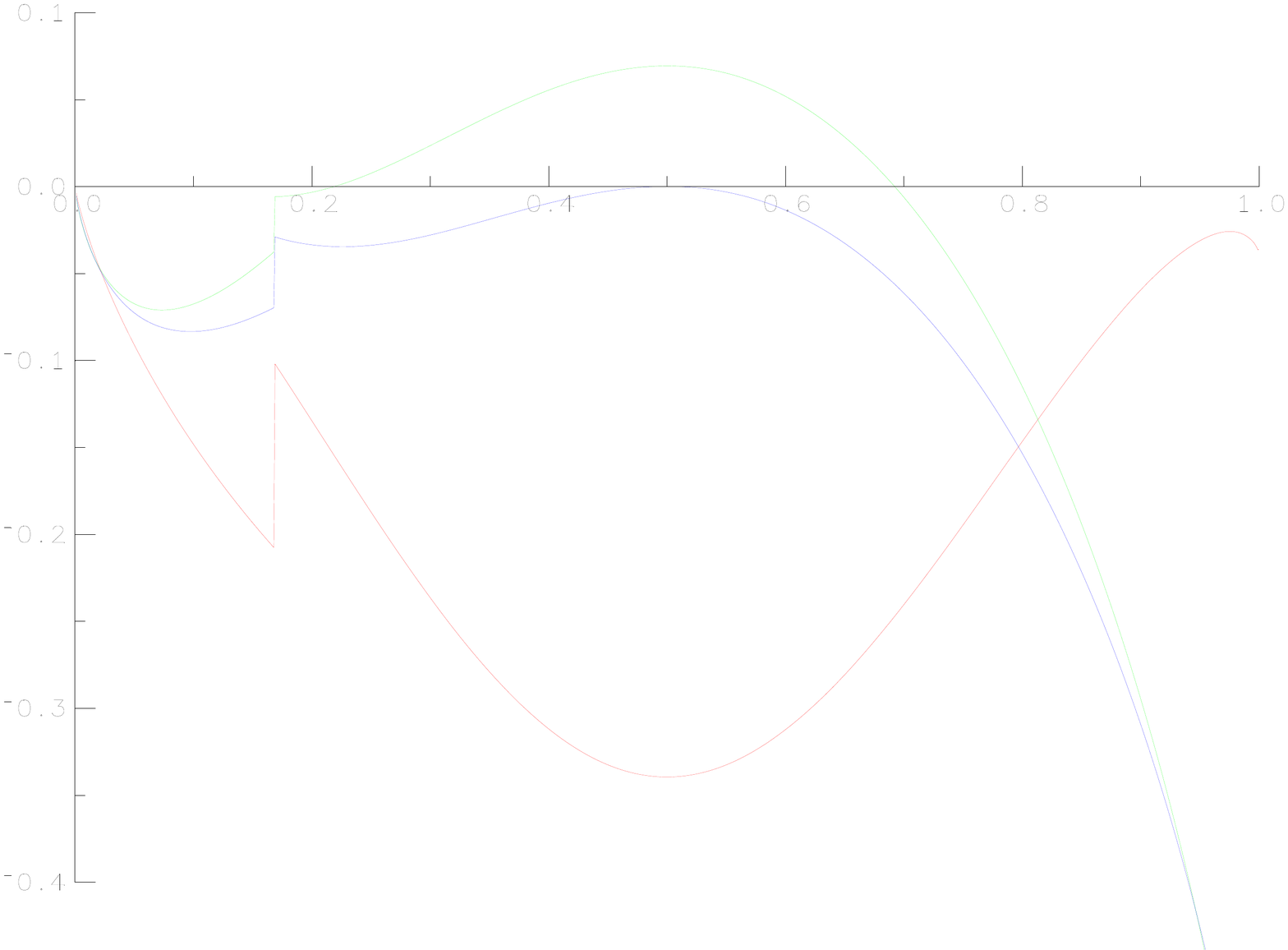}
}\else{%
\includegraphics[width=4.0in]{H4_alphak_crop.eps}
}\fi
\end{centering}
\caption{
Plot of $\HH$ versus $\a$, for $k=4$ and $c$ values of
0.51, 1.0, 1.1 (from bottom to top).
The kinks occur at $\ak$, where we switch functional forms for
$\zzzargs$.
}
\label{H4}
\end{figure}

\begin{claim} \label{asmall}
For all $k \geq 3$ and all $c \in (2/k, 1]$,
taking
\begin{align}
\z_1= (ck)^{-1/2} \a^{1/2}, \quad \z_2=\ab \label{zsmall}
\end{align}
yields 
$\HH \leq (c \a)(\tfrac k 2-1) \ln ( \a / \ak)$
for all $\a \in (0,\ak)$.
Also, for any $\d=\d(k)>0$ 
there exists $\e=\e(k)>0$ such that
$\HH < -\e$
for all $\a \in [\d,\aak]$.
In both cases, 
$\z_2 \geq \zn(k) := 1-\ak>0$. 
\end{claim}

The first part of the claim establishes \eqref{lemma1.2},
and the second part, with $\d=\ak/3$, establishes \eqref{lemma1.1} 
for $\a \in [\ak/3,\aak]$.
As both $\e$ and $\zn$ depend only on $k$ 
they are automatically continuous (constant) with respect to $c$,
thus satisfying the hypothesis of Lemma~\ref{Hlemma}.

\begin{proof}
Trivially,
$\z_2 = \ab \geq 1-\ak>0$,
since $\ak = e k^{-k/(k-2)} < e/k < 1$.
The issue in this range of $\a$ is 
to control the final logarithmic term of $\HH$
when the two summands within the logarithm are nearly equal.
Note that $\ln f(x)$
is concave on either side of 0
(diverging to $-\infty$ at $0$, it is not concave as a whole),
as
\begin{align}
\bigl[\ln f(x)\bigr]''=\frac{e^x(1-x-e^{-x})}{f^2(x)}<0.
 \label{flogconcave}
\end{align}
Since 
$\left. \frac d{d\D} \ln f(\l (1+\D)) \right|_{\D=0} 
 = \frac{\l f'(\l)}{f(\l)}
$,
if $\l$ and $\l\tcdot(1+\D)$ are on the same side of 0 (i.e., if $1+\D \geq 0$)
then concavity gives
$ \ln f(\l (1+\D))
  \leq \ln f(\l) + \D \frac{\l f'(\l)}{f(\l)} $.
Or, with $\z=1+\D$, if $\z \geq 0$ then
\begin{align}
\frac{f(\l \z)}{f(\l)}
  \leq \exp\left((\z-1) \frac{\l f'(\l)}{f(\l)}\right)
  = \exp\left((\z-1) ck \right) ,
  \label{fratioexpbound}
\end{align}
recalling from \eqref{lambda} 
that $\l f'(\l)/f(\l) = ck$.
It is easily checked that \eqref{ak} gives $\ak < 0.2$,
hence from \eqref{zsmall}
$\z_2>0.8$ and $\z_1<0.4$, so
$\z_2-\z_1 \geq 0$ and of course
$\z_2+\z_1 \geq 0$.
Thus for the final term of $\HH$, from \eqref{fratioexpbound} we have
\begin{align*}
 \hspace*{0.5in}&\hspace*{-0.5in}
 \ln\frac{f(\l\tcdot(\z_2+\z_1))+f(\l\tcdot(\z_2-\z_1))}{2f(\la)}
  \\ &\leq
  \ln \left(
    \frac{ \exp(ck(\z_2+\z_1-1))} 2 +
    \frac{ \exp(ck(\z_2-\z_1-1))} 2
     \right)
  \\&=
  \ln \left(
    \exp(ck(\z_2-1)) \left[
    \frac{ \exp(ck \z_1) + \exp(-ck \z_1)}{2}
     \right] \right)
  \\&=
  ck(\z_2-1) + \ln \cosh (ck \z_1)
  \\& \leq
  ck(\z_2-1) + (ck \z_1)^2/2,
\end{align*}
using the well known inequality 
$\cosh x \leq e^{x^2/2}$ 
Now also using $-\ab \ln \ab \leq \a$ for all $\ab \in [0,1)$,
substituting $\zzz$
from \eqref{zsmall} into $\HH$,
\begin{align*}
 \HH
  & \leq -c \a \ln \a + c \a + c k \a \ln ((c k \a)^{1/2}) +0
    +ck (-\a) + \sqrt{ck\a}^2/2
  \\& = 
  c \a [(\tfrac k 2-1) \ln \a + (1-\tfrac{k}2) + \tfrac k 2 \ln (c k) ]
  \\& =
  (c \a)(\tfrac k 2-1) \ln [ \a \: \tfrac1e (ck)^{k/(k-2)} ] .
\end{align*}
Pessimistically taking $c=1$ within the logarithm
and recalling $\ak$ from \eqref{ak},
\begin{align}
\HH
  & \leq
  (c \a)(\tfrac k 2-1) \ln ( \a / \ak) .
   \label{small1}
\end{align}
(A different upper bound for $c$ would simply call for a different
value for $\ak$.)
This proves the first part of the claim.

Clearly, for all $\a \in (0,\ak)$, $\a \ln(\a/\ak)$ 
is negative,
so for any $\d=\d(k)>0$, over 
$\a \in [\d,\aak]$ 
it is bounded away from 0.
By hypothesis, $c \geq 2/k$ (any positive constant would do),
thus $\HH$ is also bounded away from 0, i.e., there is some $\e=\e(k)>0$
for which $\HH \leq -\e$.
This proves the second part of the claim.
\end{proof}

\begin{claim} \label{alarge}
For all $k \geq 4$ and all $c \in (2/k, 1)$,
there exists $\e = \e(c,k)>0$, 
with $\e(c,k)$ continuous in $c$,
such that
for all $\a \in \bigint$,
taking 
\begin{align}
\z_1=\a, \quad \z_2=\ab
 \label{usualzzz}
\end{align}
yields
$\HH < -\e$ 
and (trivially) $\z_2 \geq \zn := 1/2$.
\end{claim}

\begin{proof} 
Recall the definition of $\HH$ in \eqref{H=},
including its use of 
$\la=\la(kc)=\psiinv(kc)$, i.e., $\psi(\la)=kc$
(see \eqref{lambda}).
If we let
\begin{align}
\gala & := \frac{f(\la\tcdot(1-2\a))} {f(\la)}
\label{gdef}
\end{align}
then we have
\begin{align}
H_k(\a,\zzz(\a;c);c)
 &= cH(\a)+\ln\frac{f(\la)+f(\la\tcdot(1-2\a))}{2f(\la)}
\label{Hk2}
\\ &=
\frac{\psi(\la)}k H(\a) 
+ \ln\frac{1+\gala}2
=:\Hkala .
\label{Hk3}
\end{align}
The advantage
of $\Hkala$ over $H_k(\a,\zzz;c)$ is that the former is an explicit function
of the ``hidden'' parameter $\la=\la(ck)$, 
while the latter depends on $c$ both explicitly,
and implicitly via $\la(ck)$. 
(To put it another way, $\la$ appears repeatedly in $\HH$
and is only implicitly defined as $\psiinv$ (see \eqref{lambda}),
where $\psi$ appears just once in $H_k$ 
and is explicitly defined (see \eqref{fdef}).)

Since $\la(\cdot)$ is increasing (see after \eqref{lambda})
and $c\in (2/k,1)$, 
\begin{align}
\la := \la(ck) \in 
(\la(2),\la(ck)] \subset
(0,\lak) 
 , \quad \text{ where } \quad
\lak:=\la(k) .
 \label{lakdef}
\end{align}
We now argue that it suffices to consider only the largest possible value of
$c$, namely $c=1$, or correspondingly of $\la$, namely $\la=\lak$.
Referring back to the original question about the $k$-XORSAT phase transition,
in the unconstrained model such a form of monotonicity 
is obvious: 
if random instances of given density are a.a.s.\ satisfiable, 
the same is true of sparser instances, as there is a coupling in which
we simply eliminate some constraints. 
But in the constrained model in which we are now working, 
monotonicity is not obvious: it is not clear that sparser instances
are more likely to be satisfiable than denser ones.
We attempted unsuccessfully to show this 
by converting to and from the unconstrained model.

In the next part we prove a more limited form of monotonicity,
in a short following section we show as a consequence that it suffices to 
show that $H_k(\a;\lak) \leq 0$,
and in a third part we do so.

\subsection*{Monotonicity}
For $k \geq 4$, 
there exists a $\sigma_k>0$ 
such that for all
$\a\in \bigint$ and
$\la\in [0,\lak]$,
\begin{align}
\text{if } \Hkala\ge 0 & \text{ then } 
\frac{\partial \Hkala}{\partial\la}\ge \sigma_k .
\label{mon}
\end{align}
In words, as a function of $\la$, $\Hkala$ is strictly increasing when it is non-negative. 

By \eqref{Hk3}, the condition $\Hkala\ge 0$ is equivalent to
\begin{equation}\label{Ha>}
H(\a)\ge\frac{k}{\psi(\la)}\ln\frac{2}{1+\gala}.
\end{equation} 
Also,
\begin{align}
\frac{\partial \ln \gala}{\partial \la}
 &=
 \frac{\partial \gala)/ \partial \la}{\gala}
 = 
 \frac{\partial \ln f(\la \tcdot(1-2\a))}{\partial \la}
  - \frac{\partial \ln f(\la)}{\partial \la}
\notag \\&=
 (1-2\a) f'(\la\tcdot(1-2\a))/f(\la\tcdot(1-2\a)) - f'(\la)/f(\la)
\notag \\&=
 \frac1\la \psi(\la\tcdot(1-2\a)) - \frac1\la \psi(\la) ,
 \notag
\end{align}
from which
\begin{align}
 \frac{\partial \gala)}{\partial \la}
  &=
 \la^{-1} \gala \big[ \psi(\la\tcdot(1-2\a)) - \psi(\la) \big] .
 \label{gpartial}
\end{align}

Differentiating \eqref{Hk3}, 
under the assumption that $\Hkala>0$ 
and using \eqref{Ha>} and \eqref{gpartial} in the first inequality,
\begin{align}
\frac{\partial \Hkala}{\partial \la}
 &=
 k^{-1}\psi'(\la) H(\a)+
    \frac{\partial \gala/\partial\la}{1+\gala}
\notag \\ &\ge
\frac{\psi'(\la)}{\psi(\la)}\ln\frac{2}{1+\gala}
+\la^{-1}\frac{\gala}{1+\gala}\bigl[\psi(\la(1-2\a))-\psi(\la)\bigr] 
\notag \\ &\ge
 \frac{\psi'(\la)}{\psi(\la)}
 \left[\ln\frac{2}{1+\gala}-2\a\psi(\la)\frac{\gala}{1+\gala}\right] ,
 \label{H'1}
\end{align}
where the second inequality uses that $\psi'(\la) > 0$ and,
by convexity of $\psi$ (see Claim~\ref{psi''} for both), that
$\psi(\la(1-2\a))-\psi(\la)\ge -2\la\a \psi'(\la)$.

Now, regarding $\gala$ as an independent quantity, 
the RHS of \eqref{H'1} is decreasing with $\gala$,
and for $\la \leq 1/2$
\begin{align}
\gala &\le \exp\bigl[-2\a\psi(\la)\bigr] \label{gprop}
\end{align}
since concavity of $\ln f(x)$ for $x>0$ (see \eqref{flogconcave})
means that
$\ln \gala 
 \leq -2\a\la \frac{d}{d\l}(\ln f(\la))
 = -2 \a \la \frac{f'(\la)}{f(\la)}
 = -2 \a \psi(\la)$.
It follows then from \eqref{H'1} that 
\begin{align}
\frac{\partial \Hkala}{\partial \la}
 &\geq
\frac{\psi'(\la)}{\psi(\la)}F(2\a \psi(\la))
\label{H'2}
\end{align}
where
\begin{equation*}
F(x):=\ln\frac{2}{1+e^{-x}}-\frac{x e^{-x}}{1+e^{-x}}.
\end{equation*}
Using $\ln(x) \leq x-1$ we can confirm that 
$\ln \frac2{1+e^{-x}} 
 = -\ln (\tfrac12 + \tfrac12 e^{-x}) 
 \geq -(\tfrac12 e^{-x} - \tfrac12) 
 = \frac{\sinh(x)}{1+e^x}$, 
from which
\begin{align}
F(x) &\ge \frac{\sinh x -x}{1+e^x} \ge \frac{x^3}{6(1+e^x)}
 .
\label{F>}
\end{align}
By definition (see \eqref{lambda}),
$\psi(\la)=ck$, 
and here, $x=2\a \psi(\la) = 2\a c k \leq k$.
With this, \eqref{H'2} and \eqref{F>},
\begin{align*}
\frac{\partial \Hkala}{\partial \la}
 \geq \frac{\psi'(\la)}{\psi(\la)} \; \frac{(2\a\psi(\la))^3}{6(1+e^k)}
 = \frac{\psi'(\la) (2\a)^3 (\psi(\la))^2} {6(1+e^k)}
  > 1.7 \ak^3 / (1+e^k) .
\end{align*}
For the final inequality,
calling again on Claim~\ref{psi''},
$\psi'$ is increasing, so
$\psi'(\la) \geq \psi'(0) = 1/3$.
Here we in the range $\a \geq \aak$, 
and again $\psi(\la)=ck$, which by hypothesis is $>2$.
This establishes \eqref{mon} 
with $\sigma_k:= 1.7 \ak^3/(1+e^k)$.

\subsection*{Application of monotonicity}
For $c \in (2/k,1)$ as hypothesized in the Claim, 
we will show that 
\begin{align}
m_k(c) & := 
\sup \set{\Hkala \colon \a\in \bigint, \, \la\in [0,\la(ck)]} 
 < 0.
  \label{mkneg}
\end{align}
By continuity of $\Hkala$, 
the supremum 
is attained at some $(\ahat,\lahat)$.
Recall from \eqref{lakdef} that $\la(ck) < \lak$.
By \eqref{mon}, 
if $H_k(\ahat;\lahat) \geq 0$ then
$\partial H_k(\ahat;\la)/\partial \la\ge \sigma_k$ 
for all $\la\in [\lahat,\lak]$,
implying that $H(\ahat,\lak) > 0$.
In the next part we will show that this is impossible ---
that $H(\a,\lak) \leq 0$ ---
and thus that $m_k(c)<0$.
For Claim~\ref{alarge} we may thus take $\e(c,k)= - m_k(c)$.
That this is continuous in $c$ is immediate from continuity 
of $\Hkala$.

\subsection*{Analysis of the extreme case}
The proof of the Claim is complete except for treatment of the extreme case,
$c=1$ or equivalently $\la=\lak$, namely showing that 
\begin{align} 
\Hka := H_k(\a;\lak) \leq 0
 \label{Hkadef}
\end{align}
for all $\a\in \bigint$. 
(Observe, e.g.\ from \eqref{Hk1} below, that $\Hk(1/2)=0$.)
We begin with
\begin{align}
 \Hka 
  &= H(\a)+\ln\frac{f(\lak)+f(\lak(1-2\a))}{2f(\lak)}
   \label{Hk1} \\ &= 
   H(\a)-\ln 2+\ln\left(1+\frac{f(\lak \tcdot (1-2\a))}{f(\lak)}\right)
 \notag \\ &\le
  \;-\; \frac{(1-2\a)^2}{2}\left(1-\frac{\lak^2e^{\lak(1-2\a)}}{f(\lak)}\right) ,
 \label{Hka}
\end{align}
where inequality \eqref{Hka} uses that, as $H''(\a)\le -4$, 
\begin{align}
H(\a)-H(1/2)\le -\tfrac{1}{2}(1-2\a)^2 ,
 \label{entUB}
\end{align}
and that
\begin{align}
 \ln(1+x) \le x 
 \quad \text{ and } \quad 
f(x) &= \frac{x^2}{2}\sum_{j\ge 0}\frac{2x^j}{(j+2)!}
 \le \frac{x^2e^x}{2}
 .  \label{fbound}
\end{align}

\mysubsubsec{Case $\a$ near $1/2$}
It is immediate from \eqref{Hka}
that $\Hka \leq 0$ for
$\a$ sufficiently close to $1/2$, namely for
$\a \in [\akstar, 1/2]$, where
\begin{align} \label{defakstar}
 \akstar 
  := \frac12 \left(1-\frac{1}{\lak}\ln\frac{f(\lak)}{\lak^2}\right).
\end{align}
Let us confirm that $\akstar\in (0,1/2)$, i.e., that
$\frac1{\lak} \ln(f(\lak)/\lak^2) \in (0,1)$.
First, we show that for all $k \ge 3$, $\lak>k-1$.
This is equivalent to $k > \psi(k-1)$, or
explicitly to $e^{k-1}>1+k(k-1)$, which 
follows for $k \geq 7$ by use of $e^x > 1+\tfrac16 x^3$,
and simply by checking for $k<7$.
Then, by definition, 
$k = \psi(\lak)= \lak+{\lak^2}/{f(\lak)}$, 
so $\lak>k-1$ implies that $\lak^2 / f(\lak) < 1$, 
giving $\tfrac1\lak \ln(f(\lak)/\lak^2)>0$.
Also, $\lak>k-1$ implies $\lak \geq 1$, 
from which $f(\lak)/\lak^2 \leq f(\lak) < e^\lak$,
and $\tfrac1\lak \ln(f(\lak)/\lak^2)<1$.

\mysubsubsec{Case $\a$ away from $1/2$}
We now treat 
$ \a \in [\aak,\akstar]  $
through two sub-cases.

\mysubsubsec{Subcase $k\ge 7$}
Since $\frac{f(\lak \tcdot (1-2\a))}{f(\lak)} = \gala \leq e^{-2\a k}$
(by \eqref{gdef} and \eqref{gprop},
the latter relying on $\a \leq \akstar \leq 1/2$), 
we have from \eqref{Hk1} that
\begin{equation}\label{tildeHalphak}
\Hka<
H(\a)+\ln\frac{1+e^{-2 \cdot \a k}}{2} 
 \leq 
H(\akstar)+\ln\frac{1+e^{-2 \cdot \aak k}}{2} .
\end{equation}

Let us show that $\akstar$ decreases with $k$,
implying that 
$H(\akstar) \leq H(\astar_7)$.
Since $\la(\cdot)$ is increasing (see after \eqref{lambda}),
it suffices to show that $\tfrac1x \ln (f(x)/x^2)$ increases with $x$ 
for $x \geq \l_3$; we will show it for all $x>0$.
Differentiating,
$$
\frac d{dx} 
\left(\frac{1}{x}\ln\frac{f(x)}{x^2}\right) = -\frac{1}{x^2} \, G(x)
\quad \text{where} \quad
G(x):=\ln \frac{f(x)}{x^2}+2-\frac{xf'(x)}{f(x)} ,
$$
so we must show that $G(x)<0$ for $x>0$.
Now, $\lim_{x\downarrow 0}G(x)=-\ln 2<0$, 
so it suffices to show that
$G'(x) = \bigl(\psi(x)-2-x\psi'(x)\bigr)/x \leq 0$,
or equivalently
$\psi(x)-2-x\psi'(x) \leq 0$.
This is true, since this expression is $0$ at $x=0$
and its derivative is simply $-\psi''(x)$, which is $\leq 0$ by
convexity of $\psi$ (see Claim~\ref{psi''}).

Also, recalling the definition of $\ak$ from \eqref{ak},
differentiation immediately shows that $\ak k = e k^{-2/(k-1)}$
increases with $k$, so that 
$\ak k \geq 7 \, \a_7$.

So, for $k\ge 7$ and $\a \in \botint$, 
\eqref{tildeHalphak} yields the cruder bound
\begin{equation}\label{simpler}
\Hka < 
  H(\astar_7)+\ln\frac{1+e^{-2 \cdot 0.99 \cdot 7 \, \a_7}}{2}
  < -0.019 .
\end{equation}

\mysubsubsec{Subcase $k=4,5,6$} 
Notice that, for $\a\in [0,1/2]$, 
the entropy term $H(\a)$ in \eqref{Hk1} for $\Hka$ is increasing, 
while the logarithmic term is decreasing. 
Consequently, if $0<x<x'\le 1/2$ are such that
\begin{equation}\label{Hkint}
H(x')+\ln\frac{f(\lak)+f(\lak \tcdot (1-2x))}{2f(\lak)}<0,
\end{equation}
then $\Hka<0$ for all $\a\in [x,x']$. 
A collection of such intervals $[x,x']$ covering $[\aak,\akstar]$
gives an ``interval arithmetic'' proof that 
$\Hka \leq 0$ on $[\aak,\akstar]$,
and there is an elegant iterative procedure for finding such a cover.

The LHS of \eqref{Hkint} is equal to
\begin{align}
 \Hk(x)+H(x')-H(x) 
  & < \Hk(x)+(x'-x) H'(x)
  \label{Hkint2}
\end{align}
by convexity of $H$.
Thus, inequality \eqref{Hkint} is satisfied
if the RHS of \eqref{Hkint2} is 0,
i.e., if 
\begin{align} \label{iterate}
x' &= x-\frac{\Hk(x)}{H'(x)}
\end{align}
(Note that $\Hk(x)<0$, so $x'>x$.)
We apply \eqref{iterate},
reminiscent of Newton-Raphson,
as an iterative update rule,
with $x_i=x$ and $x_{i+1}=x'$,
to cover the interval $[\aak,\akstar]$.

For $k=6$, taking $x_0 = x = \aak \approx 0.1831$ 
gives $x_1 =x' \approx 0.2620$, 
showing that $H_6(\a) \leq 0$ on $[x_0,x_1]$.
Then, taking $x=x_1$ gives $x_2 =x' \approx 0.3421$,
showing that $H_6(\a) \leq 0$ on $[x_1,x_2]$.
Since $\astar_6 < 0.3024 < x_2$,
for $k=6$ these two intervals suffice to prove negativity of $\Hka$
over $[\aak,\akstar]$.

For $k=5$, following the same procedure covers $[\aak,\akstar]$
with 3 intervals. 
Likewise, for $k=4$, $[\aak,\akstar]$ is covered with 8 intervals.

In fact, \eqref{simpler} and a check of the intervals for $k=4,5,6$
yields that, for $k \geq 4$,
\begin{align} \label{midrange}
\Hka &< -0.0012 \quad (\forall \a \in \botint). 
\end{align}
This completes the proof of Claim~\ref{alarge}. 
\end{proof}

We have not addressed $k=3$, already treated by \cite{DM02},
and indeed with $\zzz$ as above, 
$H_3(\tfrac12,\zzz;1)$ is positive.
We remark that we can extend Claim~\ref{alarge} to $k=3$ by choosing
$\zzz$ differently, notably as given by \eqref{barzs}.
The motivation is that the equalities \eqref{barzs}
hold for the optimal $\zz=\la \zzz$
at the stationary points $(\ahat,\hat k)$ of the function
$\min_{\zz} \Hk(\a, \zz/\la; \psi^{-1}(ck))$,
assuming (without justification) that the implicit-differentiation rules
apply.
The monotonicity condition (the equivalent of \eqref{mon}) 
then applies for all $\a \in (0,1/2]$.
For details, see \cite[Appendix (b)]{ArxivBoris}.
An interval arithmetic argument verifies that this choice makes
$H_3(\a,\zzz;1) < 0$ for $\a \in [0.99\a_3,\astar_3]$, 
as we will show after \eqref{barzs}
where this is needed to treat the phase transition more precisely.
If we make this extension,
Claim~\ref{ahalf} also extends immediately to $k=3$.

We also remark that 
if we alter the hypotheses of Claim~\ref{alarge}
to exclude $\a$ near $1/2$ then we may allow $c=1$,
as formalized below (where the choice of $0.49$ is arbitrary).
This will be used when we narrow the phase transition window
in Section~\ref{sharper}.
\begin{remark} \label{RemAway}
For all $k \geq 4$ 
there exists $\e = \e(k)>0$, 
such that
for all $\a \in \awayint$ 
and all $c \in [2/k, 1]$,
taking 
$\z_1=\a, \quad \z_2=\ab$
yields
$\HH < -\e$ 
and (trivially) $\z_2 \geq \zn := 1/2$.
\end{remark}

\begin{proof}
The substitution gives $\HH=\Hkala$ (see \eqref{Hk3}),
the range $c \in [2/k,1]$ corresponds to $\la \in [0,\lak]$,
and it suffices to show that 
$$ \sup \set{\Hkala \colon \a\in \awayint ,\, \la\in [0,\la_k]} 
 < 0 .
$$
In analogy with \eqref{mkneg},
by continuity, the supremum over this closed domain is achieved at some
$(\ahat,\lahat)$.
We prove by contradiction that $\Hkhat<0$.
If not, $\Hkhat \geq 0$.
If $\lahat<\lak$ then as argued previously this implies 
$\Hk(\ahat,\lak) = \Hk(\ahat)>0$,
while if $\lahat=\lak$ then, directly, $\Hk(\ahat) \geq 0$.
We now show that $\Hka<0$ for $\a \in \awayint$, by modifying the
previous argument 
that $\Hka \leq 0$ for $\a \in \bigint$.
Referring to \eqref{defakstar}, for any $\akp$ with $\akstar<\akp<0.49$,
inequality \eqref{Hka} shows that 
$\Hka<0$ for $\a \in [\akp,0.49]$.
(Recall that $\akstar$ is decreasing in $k$ 
--- see after \eqref{tildeHalphak} ---
so for all $k \geq 3$, $\akstar \leq \astar_3 < \akpt$.)
And from \eqref{simpler},
continuity shows that for some $\aplus_7$ slightly larger than $\astar_7$
we have $\Hka< -0.018$ for all $\a \in [\aak,\aplus_7]$ and $k \geq 7$.
Likewise, for the numerically treated cases $k=4,5,6$,
\eqref{midrange} extends by continuity to show that, 
for some $\akp$ slightly larger than $\akstar$,
$\Hka< -0.0011$ for all $\a \in [\aak,\akp]$.
\end{proof}

\begin{claim} \label{ahalf}
For all $k \geq 4$ and all $c \in (2/k, 1)$,
there exist $\e = \e(c,k)>0$ and $\zn = \zn(c,k) >0$,
both functions continuous in $c$, such that
for all $\a \in (1/2,1]$
there exists $\zzz$ for which 
$\HH < -\e$ and $\z_2 \geq \zn$.
\end{claim}

\begin{proof}
For any $x >0$, $f(x) >f(-x)$;
this follows from $f(x)-f(-x)=e^x-e^{-x}-2x=2(\sinh(x)-x) > 0$,
the last inequality well known.
This gives
$$
\lim_{\a \to 1} H_k(\a,(\a,\ab);c)=\ln\frac{f(\la)+f(-\la)}{2f(\la)}<0 ,
$$
the equality immediate from 
\eqref{Hk2}
and the inequality from
$ck>2$ and thus $\la=\la(ck) >0$.
By continuity of $H_k(\a,(\z_1,\z_2);c)$ 
with respect to $\a$, $\z_1$ and $\z_2$,
there exist functions
$\delta=\delta(c,k)>0$ and $\e=\e(c,k)>0$,
both continuous in $c$, for which 
\begin{equation} \label{alpha1}
\sup_{\a\in [1-\delta,1]}
H_k(\a,(1-\d,\d);c)
 \leq -\e .
\end{equation}
This establishes the claim for 
$\a\in [1-\delta,1]$.

For $\a \in (\frac12,1-\d)$,
let $\zzz=(\z_1,\z_2)$
be given by $\z_1(\a)=\z_2(\ab)$, the latter 
determined by Claims \ref{asmall}--\ref{alarge}, and likewise 
$\z_2(\a)=\z_1(\ab)$.
Then, 
\begin{align*}
H_k(\a,\zzz(\a);c) 
 &= H_k(\a,(\z_2(\ab),\z_1(\ab));c) 
 \\& \leq H_k(\ab,(\z_1(\ab),\z_2(\ab));c) 
 = H_k(\ab,\zzz(\ab);c) .
\end{align*}
The 
inequality follows from \eqref{H=}:
for the first three terms of its right hand side by symmetry, 
and for its last term by applying the inequality $f(x) \geq f(-x)$,
with $x:=\z_2(\ab)-\z_1(\ab) \geq 0$
(in the proofs of Claims \ref{asmall}--\ref{alarge}, $\z_2 \geq \z_1$).
It follows that
\begin{align*}
H_k(\a,\zzz(\a);c) 
 \leq -\e,
\end{align*}
where $\e$ is chosen as the minimum of corresponding values 
in Claims \ref{asmall}--\ref{alarge}.
(Actually, 
here we need the value of $\d(c,k)$ chosen for \eqref{alpha1}
rather than the $\d(k)$ used in Claim~\ref{asmall}.
This goes through without difficulty since
$\d(c,k)$ is continuous in $c$, and
the corresponding $\e(c,k)$ needed in the last paragraph of 
the proof of Claim~\ref{asmall} 
is continuous in $\d$, and has no dependence on $c$ other than through $\d$.)

Finally, for $\zn(c,k)>0$ suitably chosen, we have
$\z_2 \geq \zn(c,k)>0$.
This follows because 
for $\a \geq 1-\d$ we have $\z_2=\d$, while
for $\a \in (1/2,1-\d)$ we have 
$\z_2(\a)=\z_1(\ab)$,
which by Claims \ref{asmall}--\ref{alarge} is
variously of order $\Theta(\ab^{1/2})$ or $\Theta(\ab)$,
and in either case bounded away from 0 since $\ab \geq \d$.
\end{proof}

This completes the claims used in proving Lemma~\ref{Hlemma}.

\section{More precise threshold behavior} \label{sharper}
With relatively little additional work, 
we can prove the prove the finer-grained threshold behavior
given by Theorem~\ref{sharp}.

\begin{proof}[Proof of Theorem~\ref{sharp}]
By a standard and general argument we may assume that $m/n$ has a limit.
We reason contrapositively.
If there is a sequence of $m$ and $n$
for which the desired probability 
(of satisfiability or unsatisfiability as the case may be) 
fails to approach 1 as claimed,
then it has a subsequence for which the probability approaches a 
value less than~1,
it in turn has a sub-subsequence for which $\lim m/n$ exists,
and by hypothesis it satisfies $2/k < \lim m/n \leq \infty$.
That is, if there is a counterexample, 
then there is one in which $m/n$ has a limit.
The case $\lim m/n \neq 1$ was already treated by Theorem~\ref{main},
so we assume henceforth that $\lim m/n=1$.

The unsatisfiable part of the theorem is immediate from Remark~\ref{>1}.

For satisfiability, we have $c=m/n<1$ and $c \to 1$.
We follow the outline of the proof of Theorem~\ref{main}.
Claim \ref{asmall} 
already treats $m/n=c$ in a closed interval including 1 and all $k \geq 3$.
(As Dubois and Mandler did not derive this sharper threshold,
here we must treat $k=3$.)
So does Claim~\ref{ahalf}, in its treatment of $\a$ near 1
and the symmetry argument elsewhere,
contingent upon Claim~\ref{alarge}.
That is, the previous analysis 
(encapsulated in the proof of Corollary~\ref{dupCor})
continues to apply to all terms in the sum
$
\sum_{\ell=2}^{m}\ex\bigl[\Ymnl\bigr]
$
\emph{except} those with $\ell/m = \a \in [\akp,1/2]$,
so that, in the current setting,
$$
\sum_{\ell=2}^{m}\ex\bigl[\Ymnl\bigr]
=O\bigl(m^{-(k-2)}\bigr)
 + 2 \sum_{\ell=\aak m}^{m/2}\ex\bigl[\Ymnl\bigr] .
$$

To prove Theorem~\ref{sharp} we split the final summand above
into two ranges, and will show that
\begin{align}
 \sum_{\ell=\aak m}^{\akpt m}\ex\bigl[\Ymnl\bigr] = O(n) e^{-\Omega(n)} 
 \label{sum1}
\end{align}
(this will be immediate from \eqref{Hkawayneg})
and 
\begin{align}
 \sum_{\ell=\akpt m}^{m}\ex\bigl[\Ymnl\bigr] 
  = O(1) \, \exp(-0.69 \: \omegan) 
 \label{sum2}
\end{align}
(shown in \eqref{sharpsum}).
Both of these require extending Claim \ref{alarge}
to the case where $c \to 1$ (no longer bounded away from~1),
deriving fresh bounds for $\a \in \bigint$,
$k \geq 3$.
The second requires additionally an extension of Lemma~\ref{expbounds}
through an improvement,
for $\z_1$ bounded away from 0,
to inequality \eqref{abless} 
and in turn to \eqref{Ymnlless} and \eqref{YlesseH}.

The monotonicity approach used to prove Claim~\ref{alarge},
allowing us to focus on $c=1$ (correspondingly, $\la=\lak$)
no longer applies because, 
with a vanishingly small gap between $\la$ and $\lak$,
the argument no longer bounds $\Hkala$ away from 0
(indeed we already remarked that $H_k(1/2,\lak)=0$).
In lieu of the use of monotonicity, though, 
as noted above
we can assume that $c$ is less than but arbitrarily close to 1 
(correspondingly, $\la < \lak$ is arbitrarily close to $\lak$).
We now consider the two ranges of $\a$ corresponding to the 
sums in \eqref{sum1} and \eqref{sum2}.

\mysubsubsec{Case $\a$ away from $1/2$} 
We will show that, for $k \geq 3$
and an appropriate $\zzz=\zzz(\a;c)$,
$\HH$ is bounded below 0 for $c$ sufficiently close to 1
and for $\a$ in a range extending 
above $\akstar$.
Specifically, we will show that
for all $k \geq 3$
there exist $\cm<1$ and $\e(k)>0$ such that 
\renewcommand{\akp}{\akpt} 
\begin{align}
\big(\forall \a \in \botintp \big) 
\,
\big(\forall c \in [\cm,1] \big) 
 \colon \quad 
 \HH & \leq -\e(k) . 
 \label{Hkawayneg}
\end{align}
For $k \geq 4$ this was established in Remark~\ref{RemAway}.
For $k=3$ we set
\begin{align}\label{barzs}
\zeta_1 &= \exp\bigl(-H(\a)/k \bigr) \, \a^{\frac{k-1}{k}},
&
\zeta_2 &= \exp\bigl(-H(\a)/k \bigr) \, \ab^{\frac{k-1}{k}} .
\end{align}
(For more on this choice see the discussion after \eqref{midrange}.)
We have $0.0990 < 0.99 \a_3$ and $\astar_3< \akpt$,
and using interval arithmetic we verify that for subintervals
on integral multiples of $0.0001$, that is
$[0.0990,0.0991],$ $\ldots,$ $[0.4629,0.4630]$,
the value of $\HHone$ on each subinterval is $<-0.0004$.
The interval arithmetic verification consists of 
defining $\zzz$ according to the interval's first endpoint,
then considering 
the extreme values of the possible results
in each monotone component calculation for $\HHone$ (see \eqref{H=})
to get rigorous lower and upper bounds on the 
true value anywhere in the interval.
Continuity in $c$ then gives \eqref{Hkawayneg}
for some $\cm$ sufficiently close to~1.

Inequality \eqref{sum1} is immediate
from \eqref{Hkawayneg} and \eqref{YlesseH}.

\mysubsubsec{Case $\a$ near $1/2$} 
For the remaining interval $\topintp$, 
$\HH$ is not bounded away from 0,
but we will establish a sufficient bound.
We again take $\zzz=(\a,\ab)$, 
so that $\HH=\Hkala$ (including for $k=3$).
Then, for all $k \geq 3$, 
for some $\cm<1$,
\begin{multline}
 \label{Hbeta}
\big(\forall \a \in \topintp \big) \, 
\big(\forall c \in [\cm,1] \big) 
 \colon \quad 
  \\ 
  \HH =
  \Hkala \leq 
 (c-1)  \Hap-\betak(1-2\a)^2 .
\end{multline}
To see this we follow the same reasoning as for \eqref{Hka},
including use of \eqref{entUB} and \eqref{fbound}
for the first inequality below:
\begin{align}
\Hkala
 &= 
 cH(\a)+\ln\frac{f(\la)+f(\la(1-2\a))}{2f(\la)}
 \notag \\ &=
 (c-1)H(\a) + (H(\a)-H(1/2))
       +\ln\left(1+\frac{f(\la(1-2\a))}{f(\la)}\right)
 \notag  \\ & \leq
 (c-1) \Hap - \frac{(1-2\a)^2}{2}
  \left(1-\frac{\la^2e^{\la(1-2\a)}}{f(\la)}\right) .
   \notag 
 \intertext{%
Since the derivative of $\frac{\la^2e^{\la(1-2\a)}}{f(\la)}$
is bounded uniformly over $\a$,
and making no presumption about the sign of the $O(\cdot)$ term,
this is}
    &=
  (c-1) \Hap -\frac{(1-2\a)^2}{2} 
    \left(1-\frac{\lak^2e^{\lak(1-2\a)}}{f(\lak)}+O(\la-\lak)\right)
 \label{Hbeta2}
\end{align}
Now observe that for $\a \in [\akpt,1/2]$,
using the definition \eqref{defakstar} of $\akstar$,
\begin{align*}
 \frac{\lak^2e^{\lak(1-2\a)}}{f(\lak)}
  &=
 \frac{\lak^2e^{\lak(1-2\akstar)}}{f(\lak)} \; e^{2\lak (\akstar-\a)}
  =
 e^{2\lak (\akstar-\a)}
 \\ & \leq
 e^{2 \la_3 (\astar_3 - \akpt)}
 \leq 
 0.9977.
\end{align*}
This and \eqref{Hbeta2} yield \eqref{Hbeta}.

\newcommand{\szzyy}{\sqrt{z_1^2-y^2}\,}
\newcommand{\abscosh}{| {\cosh (z_1 e^{i\vt}) -1} |}

\mysubsubsec{Improved bound on $\alnu$} 
We will 
need bounds on $\alnu$ and $\ex[\Ymnl]$
better than those in
\eqref{abless} and \eqref{Ymnlless}.
Reasoning as for \eqref{blessbetter}, from \eqref{aseries} we have
\begin{align}
\alnu
 &=
 \frac{(k\ell)!}{2\pi}\!\!
 \oint\limits_{{\mbox{\Large${z=z_1e^{i\vt} \colon 
     \atop\vt\in (-\pi,\pi]}$}}\!\!}
 \frac {(\cosh z -1)^{\nu}} {z^{k\ell+1}} \, dz
 \notag \\ &\leq
 2 \, \frac {{(k\ell)!}} {2\pi \, z_1^{k\ell+1}} \!\!
  \int_{-\pi/2}^{\pi/2}
     \exp\parens{ \nu \; \ln\abscosh }
 \, d\vt 
 \notag
 \\&=
 O(1 / \sqrt \nu) \:  {(k\ell)!}
 \frac {(\cosh z_1 -1)^{\nu}} {z_1^{k\ell+1}} .
 \label{a3}
\end{align}
The final equality is by the Laplace method for integrals;
see for example de Bruijn~\cite{deBruijn}. 
Roughly, the Laplace method says that
if $f(x)$ is maximized on $[a,b]$ by $x_0$ then, asymptotically in $n$,
$\int_a^b e^{n f(x)} dx 
 = (1+o(1)) e^{n f(x_0)} \sqrt{\frac{2\pi}{n(-f''(x_0))}}
$.
The maximum of $\abs{\cosh(z_1 e^{i\vt}-1}$ 
occurs iff $\vt$ is a multiple of $\pi$,
as is clear from the Taylor series expansion
$\cosh(z_1 e^{i\vt}-1)
 = \sum_{j=1}^\infty \frac1{(2j)!} {z_1}^{2j} e^{i \cdot 2j\vt}$.
The modulus of this expression is
 $\sum_{j=1}^\infty \frac1{(2j)!} {z_1}^{2j}$
when $\vt$ is multiple of $\pi$, and only then
(for this to be the case all the arguments $2j\vt$ must be equal modulo $2\pi$,
requiring $\vt$ to be a multiple of $\pi$).
In the range $[-\pi,\pi]$, then, the unique maximum is at $\vt=0$.
Letting
\begin{align*}
 s(\vt)
  & := \ln\abscosh 
  =  \ln( \cosh(z_1 \cos \vt) - \cos(z_1 \sin \vt)) ,
\end{align*}
the second derivative at the maximum is\begin{align*}
\left. \frac{d^2 s}{d\vt^2} \right|_{\vt=0} 
 &= - \frac{z_1 (\sinh(z_1)-z_1)}{\cosh(z_1)-1}
 = -\Theta(1),
\end{align*}
since $z_1=\Theta(1)$.

\mysubsubsec{Improved bound on $\ex[\Ymnl]$}
Note that the bound on $\alnu$ in \eqref{a3} is $\nuf$ times 
the previous bound given by \eqref{abless},
since $z_1 = \z_1 \la$ 
and we presume throughout that $\la$ is bounded away from 0,
which represents an improvement when $\z_1=\Theta(1)$.
It immediately gives a corresponding improvement
to the bound on $\plnu$ from \eqref{peq}:
where $\Tnu$ represents the RHS of \eqref{peq}
(the notation focuses on the parameter of interest,
but recall that $\Tnu$ also depends on $z_1$, $z_2$, $\la$, $k$, $m$ and
$\ell$), 
we now have
\begin{align}
 \label{pnu}
  \plnu  &= \nuf \, \Tnu .
\end{align}
This improves the summands of \eqref{EYmn1},
but the $1/\sqrt \nu$ stops us from applying the binomial theorem
to obtain an analog \eqref{Ymnlless};
one additional step is needed.

As we did for \eqref{Ymnlless}, restrict $z_1$ and $z_2$
to depend only on $\ell$, $m$ and $n$ (not on $\nu$).
Then the maximum of $\nuf \, \Tnu$
can be seen to occur where the ratio of consecutive terms,
$$
\frac{(1/\sqrt{\nu+1}) \, \Tnux{\nu+1}}{(1/\sqrt \nu) \, \Tnu} =
\sqrt\frac{\nu}{\nu+1} \;\; \frac{n-\nu}{\nu+1} \;\;
 \frac{e^{z_2}(\cosh z_1-1)} {f(z_2)} ,
$$
is 1, which occurs
at some $\nu_0=\Theta(n)$.
Terms before $\nu_0/2$ are exponentially smaller than the maximum,
while later terms are of order 
$\nuf \, \Tnu$ $=$ $\nf \, \Tnu$.
Thus, 
\begin{align*}
 \sum_{\nu=1}^n & \binom n \nu  \plnu 
 \\ & = 
   O(1) \sum_{\nu=1}^{\nu_0/2} \binom n \nu (1/(z_1 \sqrt \nu)) \, \Tnu
 + O(1) \sum_{\nu=\nu_0/2}^n \binom n \nu (1/(z_1 \sqrt \nu)) \, \Tnu
 \\& = 
   O(n) \exp(-\Theta(n)) \Tnux{\nu_0}
   + \nf \sum_{\nu=1}^n \binom n \nu \, \Tnu 
 \\& = 
  \nf \sum_{\nu=1}^n \binom n \nu \, \Tnu ,
\end{align*}
an analog of \eqref{EYmn1} but smaller by $\nf$.
To this we can apply the binomial theorem, as we did to \eqref{EYmn1},
giving a corresponding improvement to \eqref{Ymnlless}
and in turn \eqref{YlesseH}, namely
\begin{align}
\ex\bigl[\Ymnl\bigr]
  & \leq
 O(1) \frac1{\z_1 \, \sqrt{n \, \z_2}} \,
 \exp\bigl[n H_k(\a,\zzz;c)\bigr],
\quad \forall \zzz>0 .
 \label{EYcrit}
\end{align}

 From \eqref{Hbeta} and \eqref{EYcrit},
\begin{align}
\sum_{\ell=\akp m}^{m/2}
 \ex[\Ymnl]
 & \leq
\sum_{\ell=\akp m}^{m/2}
 \Ortn \exp\parens{ n \HH}
 \notag \\ & \leq
\Ortn 
 \sum_{\ell=\akp m}^{m/2}
 \exp\parens{ 
 (c-1) \, n \, H(\akp)
  -\betak(1-2 \ell/m)^2 n 
  }
  \notag \\ & \leq 
\Ortn \exp\parens{(m-n) H(\akp)} 
 \sum_{x=0}^{\infty}
 \exp\parens{ - \betak (x/m)^2 n} 
 \notag \\ &= 
\Ortn \exp(-H(\akp) \: \omegan) \; O(\sqrt n)
 \notag \\ & 
 = O(1) \, \exp(-\sharpconst \: \omegan)   \label{sharpsum}
 .
\end{align}
This establishes \eqref{sum2} and concludes the proof of Theorem~\ref{sharp}.
\end{proof}

\newcommand{\gk}{g_k}
\newcommand{\gkx}{\gk(x)}
\newcommand{\chat}{\hat{c}}
\newcommand{\core}{H_2}
\newcommand{\mus}{{\mu^*}}

\section{\texorpdfstring%
{Satisfiability threshold for unconstrained $k$-XORSAT}%
{Satisfiability threshold for unconstrained k-XORSAT}}
\label{Unconstrained}
If a variable appears in at most one equation,
then deleting that variable, along with the corresponding equation if any,
yields a linear system that, clearly, 
is solvable if and only if the original system was.
Stop this process when each variable appears in at least two equations,
or when the system is empty.
Dubois and Mandler analyzed unconstrained $3$-XORSAT by analyzing this
process, which ends with a (possibly empty) constrained 3-XORSAT instance.

Regarding each variable as a vertex and each equation as a hyperedge
on its $k$ variables yields 
the $k$-uniform ``constraint hypergraph'' underlying a $k$-XORSAT instance.
The process described simply restricts the instance to the 2-core of its
hypergraph.
The analysis by Dubois and Mandler for 3-XORSAT is easily generalized
to $k$-XORSAT using the (later) analyses of the 2-core of
a random $k$-uniform hypergraph, 
and we take this approach.

Note that the our (unconstrained) $k$-XORSAT model really corresponds to a 
random $k$-uniform \emph{multi}-hypergraph.
However, the probability that a random matrix corresponds to a simple graph is
${\binom n k}_{(m)} \, / \, {\binom n k}^m = 1-O(n^{-(k-2)}) = 1-o(1)$.
Thus any a.a.s.\ property of a simple random hypergraph
is also an a.a.s.\ property for random $k$-XORSAT,
and we shall proceed with the simple random hypergraph model.

It is well known that the 2-core of a uniformly random $k$-uniform hypergraph
is, conditioned on its size and order, uniformly random among all
such $k$-uniform hypergraphs with minimum degree~2.
(One short and simple proof is identical to that for
conditioning on the core's degree sequence in \cite[Claim~1]{Molloy}.)
Also, the ``core'' of a random $k$-XORSAT instance
is an instance uniformly random on its underlying hypergraph:
the (uniform) hypergraph core determines the core $A$ matrix,
while the core $b$ is simply the restriction of 
its uniformly random initial value to the surviving rows of $A$,
a process oblivious to $b$.

Thus, satisfiability of a random unconstrained instance hinges on the
edges-to-vertices ratio
of the core of its 
constraint hypergraph.

Recall the definition of $\l$ from \eqref{lambda}.

\begin{theorem} \label{unconstrained}
Let $Ax=b$ be a uniformly random unconstrained uniform random $k$-XORSAT
system with $m$ equations and $n$ variables.
Suppose that $k \geq 3$ and $m/n \to \infty$ with $\lim m/n=c$.
Define 
\begin{align*}
  \gkx := \frac{x}{k (1-e^{-x})^{k-1}} .
\end{align*}
With $\cstark = \gk(\l(k))$,
if $c<\cstark$ then $Ax=b$ is a.a.s.\ satisfiable,
and if $c>\cstark$ then $Ax=b$ is a.a.s.\ unsatisfiable.
\end{theorem}

\begin{proof}
We treat $k$ as fixed.
Restricting consideration to $x>0$,
from Molloy \cite[proof of Lemma 4]{Molloy},
$\gkx$ has a unique minimum $\chat$,
with $\gkx=c$ having no solutions for any $c<\chat$,
and two solutions for any $c>\chat$.
Simple calculus confirms
that for $k \geq 3$, $\gk$ is unimodal
(indeed, convex).

Let $H$ be a random $k$-uniform hypergraph with $m$ edges and $n$ vertices.
Molloy \cite[Theorem~1]{Molloy}
shows that if $\lim m/n<\chat$ then the 2-core is a.a.s.\ empty,
while if $\lim m/n=c>\chat$,
then with $\mu$ the larger solution of $\gk(\mu)=c$,
the order $\nn$ and size $\mm$ of the 2-core a.a.s.\ satisfy
\begin{align*}
\nn &= n \, \frac{e^\mu-1-\mu}{e^\mu} + o(n) ,
 & 
\mm &= n \, \frac {\mu (e^\mu-1)}{k e^\mu} + o(n) ;
\end{align*}
see also Achlioptas and Molloy \cite[Proposition 30]{AM}.
Actually, Molloy works in the Bernoulli model
where the number of edges of the hypergraph $H_p$ is $\Bin(\binom n k,p)$,
but the result translates to the above by standard arguments. 
Specifically, choose $p$ so that the expected number of edges of $H_p$ is $m$.
Generate a random $H$ with exactly $m$ edges as follows: 
generate $H_p$;
if it has $m$ edges or more, which with constant probability it does,
then randomly subsample $H_p$ to give $H$;
otherwise repeat.
The core of $H$ is contained in that of $H_p$, 
so $M$ and $N$ will not be larger than the bounds given for the Bernoulli
model, 
with failure probability a constant times the failure probability
of that for the Bernoulli model.
Similarly, generating $H$ by randomly augmenting an $H_p$ having
$m$ edges or fewer shows that $M$ and $N$ will not be smaller than 
the bounds given.

It follows for the core that, a.a.s.,
\begin{align}
\frac \mm \nn
 &= \frac {\mu (e^\mu-1)} {k (e^\mu-1-\mu)} + o(1) 
 = \frac1k \psi(\mu) + o(1) .
 \label{coremn}
\end{align}

Define $\mus=\la(k)$ so that $\psi(\mus)=k$;
remember from \eqref{lambda} that for $k>2$ this is well defined,
with $\mus>0$.
We claim that $\mus$ is the larger of the two values of $\mu$ for
which $\gk(\mu) = \gk(\mus)$.
Given that $\gk$ is unimodal, this is true iff $\gk'(\mus)>0$.
Now,
$$
\gk'(\mus) = \frac{1+e^{-\mus} (\mus-\mus k-1)}{k(1-e^{-\mus})^k}
.
$$
Focusing on the numerator, multiplying through by $e^\mus$,
and replacing $k=\psi(\mus)$,
this means showing that 
\begin{align*}
 e^\mus+\mus-1 - \mus
  \parens{ \frac{\mus(e^\mus-1-\mus)}{e^\mus-1} } 
 &>0 .
\end{align*}
Multiplying the expression by $e^\mus-1$ gives
\begin{multline*}
 (e^\mus+\mus-1) (e^\mus-1)
  - \mus^2 (e^\mus-1-\mus)
  \\ =
  (e^\mus-1-\mus-\tfrac12 \mus^2)^2 
   + 3 \mus (e^\mus-1-\tfrac13 \mus-\tfrac1{12} \mus^3) > 0
\end{multline*}
as desired. The inequality is immediate from the 
Taylor series for $e^\mus$, as $\mus>0$.

Let $\cstark = \gk(\mus)$.
Because $\mus$ is the larger of the two values $\mu$ for which
$\gk(\mu)=\gk(\mus)$, we may apply \eqref{coremn},
concluding that a random $k$-uniform hypergraph with 
$\lim m/n = \cstark = \gk(\mus)$
has a core where, a.a.s., $\mm/\nn = \frac1k \psi(\mus)+o(1)= 1+o(1)$.

For any $c >\cstark$, the larger solution $\mu$ of $\gk(\mu)=c$
has $\mu>\mus$ (by the unimodality of $\gk$),
and $\psi(\mu)>\psi(\mus)=k$ 
(by Claim~\ref{psi''}).
Thus, a random $k$-uniform hypergraph with 
$\lim m/n = c > \cstark$
has a core where, a.a.s., $\mm/\nn = \frac1k \psi(\mu)+o(1) > 1$.
By this section's introductory remarks 
it follows that a random $k$-XORSAT instance with
$\lim m/n = c > \cstark$
reduces to a random constrained $k$-XORSAT instance 
with $\mm/\nn$ converging in probability to a value greater than $1$,
the reduced instance is a.a.s.\ unsatisfiable,
and thus so is the original instance.

By the same token,
if $c < \cstark$ then 
either $\gk(\mu)=c$ has no solution (if $c<\chat$), or
its larger solution 
has $\mu<\mus$ 
and $\psi(\mu)<\psi(\mus)=k$.
Thus, a random $k$-XORSAT instance with
$\lim m/n = c < \cstark$
reduces to a constrained $k$-XORSAT instance that 
either is a.a.s.\ empty (and trivially satisfied),
or has $\mm/\nn$ converging in probability to a value less then $1$,
and thus is a.a.s.\ satisfiable by Theorem~\ref{main}. 
Thus the original instance is a.a.s.\ satisfiable.
\end{proof}

\section*{Acknowledgments}
We are most grateful to Paul Balister:
our proof that the critical exponent can be made negative
owes a great deal to his
detailed suggestions on using patchwork functional approximations
to get finitely away from critical points,
and interval arithmetic elsewhere \cite{Balister}.
We are also grateful to Mike Molloy for helpful comments,
to Colin Cooper, Alan Frieze, and Federico Ricci-Tersenghi 
for pointing out related work,
and to Noga Alon for suggesting we aim for Theorem~\ref{sharp}.
Finally, we sincerely thank the anonymous referees for their 
very careful reading and many helpful comments.

\end{document}